\documentclass[reqno]{amsart}
\usepackage{amssymb,amsmath,hyperref}
\usepackage{amsrefs}
\usepackage[foot]{amsaddr}
\usepackage{bbold,stackrel, tikz-cd, fixmath}
 
\newtheorem{thm}{Theorem}

\newtheorem{cor}[thm]{Corollary}
\newtheorem{defi}[thm]{Definition}
\newtheorem{rem}[thm]{Remark}
\newtheorem{nota}[thm]{Notation}
\newtheorem{exa}[thm]{Example}

\newtheorem{princ}[thm]{Principle}

\newtheorem{ack}[thm]{Acknowledgement}

\newtheorem{conj}[thm]{Conjecture}
\newtheorem*{tempo*}{Template}
\newtheorem{theorem}[thm]{Theorem}

\newtheorem{lemma}[thm]{Lemma}
\newtheorem{definition}[thm]{Definition}
\newtheorem{corollary}[thm]{Corollary}
\newtheorem{remark}[thm]{Remark}

\newcommand\be{\begin{equation}}
\newcommand\ee{\end{equation}} 

\usepackage{amsmath,amsfonts} 

\usepackage[applemac]{inputenc}

\def\bdefi{\begin{defi}\rm}
\def\edefi{\end{defi}}
\def\bnota{\begin{nota}\rm}
\def\enota{\end{nota}}
\def\FIVE{\Pi_{1}^{1}\text{-\textup{\textsf{CA}}}_{0}}

\def\SIX{\Pi_{2}^{1}\text{-\textsf{\textup{CA}}}_{0}}

\def\SIXK{\Pi_{k}^{1}\text{-\textsf{\textup{CA}}}_{0}^{\omega}}

\def\ATR{\textup{\textsf{ATR}}}

\def\Z{\textup{\textsf{Z}}}

\def\NFP{\textup{\textsf{NFP}}}
\def\ZFC{\textup{\textsf{ZFC}}}

\def\ZF{\textup{\textsf{ZF}}}

 \def\r{\mathbb{r}}

\def\RCA{\textup{\textsf{RCA}}}
\def\({\textup{(}}
\def\){\textup{)}}
\def\WO{\textup{\textsf{WO}}}
\def\RCAo{\textup{\textsf{RCA}}_{0}^{\omega}}
\def\ACAo{\textup{\textsf{ACA}}_{0}^{\omega}}

\def\WKL{\textup{\textsf{WKL}}}

\def\WWKL{\textup{\textsf{WWKL}}}
\def\bye{\end{document}}
\def\N{{\mathbb  N}}
\def\Q{{\mathbb  Q}}
\def\R{{\mathbb  R}}
\def\L{\textsf{\textup{L}}}

\def\I{{\textsf{\textup{I}}}}

\def\MUC{\textup{\textsf{MUC}}}

\def\di{\rightarrow}
\def\asa{\leftrightarrow}
\def\ACA{\textup{\textsf{ACA}}}

\def\QFAC{\textup{\textsf{QF-AC}}}

\def\HBU{\textup{\textsf{HBU}}}

\def\Y{\textup{\textsf{Y}}}
\def\PHM{\textup{\textsf{Pohm}}}
\def\HAR{\textup{\textsf{Harnack}}}
\def\BW{\textup{\textsf{BW}}}
\def\LEB{\textup{\textsf{Lebesgue}}}
\def\cocode{\textup{\textsf{cocode}}}
\def\TP{\textup{\textsf{Tp}}}

\def\DCA{\Delta\textup{\textsf{-CA}}}
\def\NCC{\textup{\textsf{NCC}}}
\def\NBIJ{\textup{\textsf{NBI}}}
\def\NBI{\textup{\textsf{NBI}}}

\def\STS{\textup{\textsf{STS}}}
\def\Arz{\textup{\textsf{Arz}}}

\def\NIN{\textup{\textsf{NIN}}}

\def\BCT{\textup{\textsf{BCT}}}

\def\w{\textup{\textsf{w}}}

\def\SS{\textup{\textsf{S}}}

\def\BOOT{\textup{\textsf{BOOT}}}
\def\open{\textup{\textsf{open}}}

\def\blambda{\pmb{\lambda}}

\def\LIN{\textup{\textsf{LIN}}}
\def\WHBU{\textup{\textsf{WHBU}}}

\def\HBC{\textup{\textsf{HBC}}}

\def\eps{\varepsilon}

\def\RT{\textup{\textsf{RT}}}

\def\ECF{\textup{\textsf{ECF}}}

\newcommand{\Tp}{\mathsf{Tp}}

\usepackage{graphicx}
\usepackage{tikz}
\usetikzlibrary{matrix, shapes.misc}

\numberwithin{equation}{section}
\numberwithin{thm}{section}

\usepackage{comment}

\begin{document}
\title{On the uncountability of $\R$}

\author{Dag Normann}
\address{Department of Mathematics, The University 
of Oslo, P.O. Box 1053, Blindern N-0316 Oslo, Norway}
\email{dnormann@math.uio.no}
\author{Sam Sanders}
\address{Department of Philosophy II, RUB Bochum, Germany}
\email{sasander@me.com}
\begin{abstract} 
Cantor's first set theory paper (1874) establishes the uncountability of $\R$.  
We study this most basic mathematical fact formulated in the language of higher-order arithmetic.  
In particular, we investigate the logical and computational properties of $\NIN$ (resp.\ $\NBI$), i.e.\ the third-order statement \emph{there is no injection \(resp.\ bijection\) from $[0,1]$ to $\N$}.   
Working in Kohlenbach's \emph{higher-order Reverse Mathematics}, we show that $\NIN$ and $\NBI$ are \emph{hard to prove} in terms of (conventional) comprehension axioms, while many basic theorems, like Arzel\`a's convergence theorem \emph{for the Riemann integral} (1885), are shown to imply $\NIN$ and/or $\NBI$.  Working in Kleene's higher-order computability theory based on S1-S9, we show that the following fourth-order process based on $\NIN$ is similarly hard to compute: for a given $[0,1]\di \N$-function, find reals in the unit interval that map to the same natural number. 
\end{abstract}
%
\maketitle
\thispagestyle{empty}

\section{Introduction}\label{intro}
By definition, the uncountability of $\R$ deals with arbitrary mappings from $\R$ to $\N$.  In our opinion, this principle is therefore best studied in a language that has such mappings as first-class citizens. 
Thus, we adopt the language of higher-order arithmetic and shall study the logical and computational properties of the uncountability of $\R$, the latter formulated \emph{in full generality} using third-order objects.
This study is therefore part of \emph{higher-order Reverse Mathematics} and Kleene's \emph{higher-order computability theory}, as explained in detail in the next sections. 
\subsection{Summary}\label{apples}
In a nutshell, we study the logical and computational properties of \emph{the uncountability of $\R$}, established in 1874 by Cantor in his \emph{first} set theory paper \cite{cantor1}, in the guise of the following natural principles:
\begin{itemize}
\item $\NIN$: \emph{there is no injection from $[0,1]$ to $\N$},
\item  $\NBI$: \emph{there is no bijection from $[0,1]$ to $\N$}.
\end{itemize}
In this paper, \emph{principle} generally refers to a statement of \emph{ordinary mathematics}\footnote{Simpson describes \emph{ordinary mathematics} in \cite{simpson2}*{I.1} as \emph{that body of mathematics that is prior to or independent of the introduction of abstract set theoretic concepts}.  The uncountability of $\R$ is studied by Simpson in \cite{simpson2}*{II.4.9}, i.e.\ the former seems to count as ordinary.}, and our aim is to investigate the logical and computational properties of these.
The principle $\NIN$ will take centre stage, while $\NBI$ will be shown to have some interesting properties as well. 
Now, a central and important aspect of mathematical logic is the classification of principles and objects in hierarchies based on logical or computational strength. 
A natural question would therefore \emph{seem to be} where $\NIN$ is located in the well-known hierarchies of logical and computational strength, generally based on \emph{comprehension} and \emph{discontinuous} functionals.  

\smallskip

We provide an answer to this question in this paper \emph{and} explain why this answer (and question) is unsatisfactory.  
Intuitively speaking, $\NIN$ is a very weak principle, yet we need rather strong comprehension axioms to prove it.  
Moreover, $\NIN$ is equivalent to restrictions of itself involving natural function classes, like semi-continuity and bounded variation (see Remark \ref{lafke}).
Thus, the logical properties of $\NIN$ are \textbf{not} due to the quantification over \emph{arbitrary} $\R\di \N$-functions in $\NIN$.   

\smallskip

Similarly, we need strong (discontinuous) comprehension functionals to compute the real numbers claimed to exist by $\NIN$ in terms of the data, in Kleene's higher-order framework.   
The reason for this paradox is that we are comparing two fundamentally different classes. 
Indeed, a fundamental division here is between \emph{normal} and \emph{non-normal} objects and principles, where the former give rise to \emph{discontinuous objects} and the latter do not (see Definition \ref{norma} for the exact formulation).  
For reference, $\NIN$ and $\NBI$ are \emph{non-normal} as they do not imply the existence of a discontinuous function on $\R$.
In this paper, all principles are part of third-order arithmetic, i.e.\ `non-normal vs normal' refers to the existence of a discontinuous function on $\R$.
The associated computations are one type-level higher. 

\smallskip

In fact, the `normal vs non-normal' distinction yield two (fairly independent) scales for classifying logical and computational strength: the standard one is the `normal' scale based on comprehension and \emph{discontinuous} objects, like the G\"odel hierarchy, Reverse Mathematics, and Kleene's quantifiers (see Section \ref{prelim}).  
However, we have shown in \cites{dagsam, dagsamII, dagsamIII, dagsamV, dagsamVI,dagsamVII} that the normal scale classifies many intuitively weak non-normal objects and principles as `rather strong'.  
We establish the same for $\NIN$ in Theorem \ref{dick} below.   
These observations imply the need for a `non-normal' scale based on (classically valid) \emph{continuity} axioms and related objects, going back to Brouwer's intuitionistic mathematics. 
The non-normal scale, and its connection to second-order arithmetic is explored in \cite{samph}, and is discussed in Section \ref{XCX}.

\smallskip

In Figure \ref{dd} below, we provide a classification of $\NIN$ and $\NBI$ relative to other non-normal principles.  
We exhibit numerous \emph{basic} theorems that imply these principles, including \emph{Arzel\`a's convergence theorem for the Riemann integral} (\cite{arse2}, 1885) and central theorems from \emph{Reverse Mathematics} (see Section \ref{prelim1}) formulated with the \emph{standard} definition of `countable set' based on injections/bijections to $\N$ (Definition \ref{standard}).
Some of these connections are made into computational results. 
As it turns out, $\NIN$ is among the weakest principles (in terms of logical and computational properties) on the non-normal scale. 
In this way, our results on $\NIN$ `reprove' many of the results in \cites{dagsam, dagsamII, dagsamIII, dagsamV, dagsamVI ,dagsamVII}, a nice bonus.
Put another way, this paper encompasses and greatly extends \cites{dagsam, dagsamII, dagsamIII, dagsamV, dagsamVI ,dagsamVII} based on perhaps the most basic property of $\R$ known to anyone with a modicum of knowledge about mathematics. 

\smallskip

We also show that theorems about countable sets can be `explosive', i.e.\ they become much stronger when combined with discontinuous functionals.  We show that the Bolzano-Weierstrass theorem for \emph{countable sets in Cantor space} gives rise to $\SIX$ when combined with higher-order $\FIVE$, i.e.\ the Suslin functional (Theorem \ref{BOOM}).  The system $\SIX$ is the
the current upper limit for RM, previously only reachable via topology (see \cites{mummymf, mummyphd, mummy}).  
Moreover, according to Rathjen \cite{rathjenICM}*{\S3}, the strength of $\SIX$ \emph{dwarfs} that of $\FIVE$, where the latter constitutes our previously `best explosion' (see Remark \ref{pimp}). 
Note that the associated Bolzano-Weierstrass theorem \emph{for sequences} in Cantor space is equivalent to $\ACA_{0}$, 
and the formulation using countable sets does not go beyond $\ACA_{0}$ \emph{in isolation}.    
We list a number of theorems about open\footnote{Open sets $O\subset \R$ in \cite{dagsamVII} are represented by $Y:\R\di \R$.  In particular, `$x\in O$' is short for $Y(x)>_{\R}0$ and $x\in O$ implies there is $n\in \N$ such that $y\in O$ for $|x-y|<\frac{1}{2^{n}}$.} sets from \cite{dagsamVII} with similar `explosive' properties.

\smallskip

While the aforementioned results are interesting to any audience of mathematicians, we also attempt to explain the underlying techniques to non-specialists, in particular Kleene's higher-order computability theory based on S1-S9 and the associated \emph{Gandy selection}.  
We sketch the historical background to this paper in Section \ref{kiko}, while a more detailed overview of our results is in Section \ref{XCX}.

\smallskip

Finally, the following principle is (potentially) stronger than $\NIN$, where \emph{weakly countable} essentially means that the set is the union over $\N$ of \emph{finite} sets. Example~\ref{tarkin} lists the exact definitions of the latter italicised notions.
\begin{itemize}
\item The unit interval $[0,1]$ is not weakly countable.  
\end{itemize}
This notion of countability only came to the fore after we finished \cite{dagsamXI}.  We have added it to this paper in light of the basic nature of this principle, in particular the definitions of `weakly countable' and `finite set'.  

\subsection{Background: Cantor and the uncountability of the reals}\label{kiko}
Georg Cantor is the pioneer of the field \emph{set theory}, which has evolved into the current foundations of mathematics $\ZFC$, i.e.\ \emph{Zermelo-Fraenkel set theory with the Axiom of Choice}; Cantor also gave us the \emph{Continuum Hypothesis}, the first problem on Hilbert's famous list of 23 open problems (\cites{hilbertendlich, hilbertlist}), which turned out to be independent of $\ZFC$, as shown by G\"odel and Cohen (\cites{cohen1, cohen2, goeset}).  
The interested reader can find a detailed account of Cantor's life and work in \cite{dauben1}. 

\smallskip

Our interest goes out to Cantor's \emph{first} set theory paper \cite{cantor1}, published in 1874 and boasting a Wikipedia page (\cite{wica}).
This short paper includes the following:  
\begin{quote}
Furthermore, the theorem in \S2 presents itself as the reason why collections of real numbers forming a so-called continuum (such as, all the real numbers which are $\geq 0$ and $\leq 1$), cannot correspond one-to-one with the collection $(\nu)$ [of natural numbers]; 
\end{quote}
This quote may be found in \cite{cantor1}*{p.\ 259} (German) and in \cite{ewa}*{p.\ 841}, \cite{grayk}*{p.\ 820}, and \cite{dauben1}*{p.\ 50}, translated to English.
Cantor's observation about the natural and real numbers may be formulated as the \emph{uncountability of $\R$}, taking into account that Cantor only introduced the notion of cardinality some years later in \cite{cantor2}.

\smallskip

Dauben provides an explanation in \cite{dauben1}*{p.\ 68-69} of why Cantor only mentions the uncountability of $\R$ in passing, as a seemingly unimportant \emph{fait divers}, in the development 
of a new proof of Liouville's theorem (on the existence of transcendental numbers).  According to Dauben, Cantor wrote \cite{cantor1} in its existing form so as to avoid rejection by Kronecker, 
one of the editors and well-known for his extreme stance against infinitary mathematics.  Weierstrass seems to have played a similar, but more moderate role, according to Ferreir\'os (\cite{nofega}*{p.\ 184}).       
In a nutshell, while results like the uncountability of $\R$ took centre stage for Cantor at the time, he deliberately downplayed them in \cite{cantor1}, so as to appease Kronecker and Weierstrass.  

\smallskip

Next, in the above quote, Cantor deduces the uncountability of $\R$ from another theorem, 
and the latter essentially\footnote{Cantor states in \cite{cantor1} that the sequence in Theorem \ref{cant2} can be given according to `any law'.} expresses the following.  
\begin{thm}\label{cant2}
For any sequence of distinct real numbers $(x_{n})_{n\in \N}$ and any interval $[a,b]$, there is $y\in [a,b]$ such that $y$ is different from $x_{n}$ for all $n\in \N$.
\end{thm}
As it happens, \emph{a lot} has been written about Theorem \ref{cant2}, its constructive status in particular.  We refer to \cite{grayk} for a detailed discussion and overview of this matter.  
We mention that \cite{grayk} includes an efficient computer program that computes the number $y$ from Theorem \ref{cant2} in terms of the other data; a proof of Theorem \ref{cant2} in a weak logical system (expressing `computable mathematics') can be found in \cite{simpson2}*{II.4.9}, while a proof in Bishop's \emph{Constructive Analysis} is found in \cite{bish1}*{p.\ 25}.  

\smallskip

In conclusion, we may safely claim that Theorem \ref{cant2} has a \emph{constructive proof}, for the various interpretations the latter term has. 
Since Cantor uses Theorem~\ref{cant2} to conclude the \emph{uncountability of $\R$} in \cite{cantor1}, it is a natural question what the logical and computational properties of the latter are, as formalised by $\NIN$ and $\NBI$.  
While of independent historical and conceptual interest, $\NIN$ and $\NBI$ shall be seen to take a central place in our ongoing project on the the logical and computational properties of the uncountable, as may be gleaned from Figure \ref{dd} in the next section.  We will observe that $\NIN$ is the most natural object of study, while (some) interesting results pertaining to $\NBI$ can be obtained.

\subsection{Logical and computational properties of the uncountability of $\R$}\label{XCX}
We sketch the results to be obtained in this paper in some detail. 
\subsubsection{Introduction: $\NIN$ and its variations}
In this section, we provide detailed (but standard) definitions of $\NIN$ and $\NBI$ as well as some conceptual discussion.  
We shall then sketch the to-be-obtained logical and computational properties of these principles relative to the `normal' scale based on comprehension and discontinuous functionals (Section \ref{norm}), as well as relative to the `non-normal' scale (Section~\ref{nonnorm}) as summarised in Figure \ref{dd}.  

\smallskip

First of all, we stress that the aforementioned notions `normal' and `non-normal' have a specific technical meaning detailed in Definition \ref{norma}.
Intuitively speaking, the normal scale is the well-known (conventional) comprehension hierarchy, while the non-normal scale is a new and independent scale.  

\smallskip

Secondly, to be absolutely clear, the uncountability of $\R$ is a statement about arbitrary mappings with domain $\R$.  Hence, the principles $\NIN$ and its ilk are \textbf{inherently third-order}, i.e.\ the below 
should be interpreted in \emph{classical}\footnote{As it turns out, there are some results on the uncountability of $\R$ in (semi-)constructive mathematics (\cites{bauer1, olipo1, diendien}).  These do not seem to relate directly to our below results.} higher-order arithmetic, namely Kohlenbach's \emph{higher-order} Reverse Mathematics (\cite{kohlenbach2}).  
Similarly, computational properties are to be interpreted in Kleene's \emph{higher-order} computability theory provided by S1-S9 (\cite{kleeneS1S9, longmann}).  These frameworks are discussed in some detail in Section \ref{prelim}.

\smallskip

Thirdly, in light of the logical and computational properties of Cantor's Theorem~\ref{cant2} from \cite{cantor1} \emph{and all the attention this has received}, one naturally wonders about the logical and computational properties of Cantor's corollary from \cite{cantor1}, namely the uncountability of $\R$.  To this end, we shall study the following principles and associated functionals as in the next section.  
\begin{princ}[$\NIN$]
For any $Y:[0,1]\di \N$, there are $x, y\in [0,1]$ such that $x\ne_{\R} y$ and $Y(x)=_{\N}Y(y)$. 
\end{princ}
\begin{princ}[$\NBI$]
For any $Y:[0,1]\di \N$, \textbf{either} there are $x, y\in [0,1]$ such that $x\ne_{\R} y$ and $Y(x)=_{\N}Y(y)$, \textbf{or} there is $N\in \N$ such that $(\forall x\in [0,1])(Y(x)\ne N)$.
\end{princ}
Finally, we stress that by Theorem \ref{flahu}, $\NIN$ can be proved \emph{without the Axiom of Choice}, i.e.\ within $\ZF$ set theory.
Hence, it is a natural question which (comprehension) axioms imply $\NIN$, as discussed in Section \ref{norm}.

\subsubsection{The uncountability of $\R$ and comprehension}\label{norm}
We discuss the logical and computational properties of $\NIN$ relative to the `normal' scale based on \emph{comprehension} and \emph{discontinuous} functionals.  
As noted in Section \ref{intro}, this is only a stepping stone towards a better picture, discussed in Section \ref{nonnorm} and summarised by Figure \ref{dd}. 

\smallskip

First of all, the logical hardness of a theorem is generally calibrated by what fragments of the comprehension axiom are needed for a proof.  
Indeed, the very aim of the \emph{Reverse Mathematics} program is to find the minimal (set-existence) axioms that prove a theorem of ordinary mathematics.  
We discuss Reverse Mathematics (RM hereafter) in some detail in Section \ref{prelim1} and note that RM-results fit in the medium range of the
 \emph{G\"odel hierarchy} (\cite{sigohi}), where this medium range is populated by fragments of second-order arithmetic $\Z_{2}$.

\smallskip

Now, Simpson studies Theorem~\ref{cant2} in \cite{simpson2}*{II.4.9}, suggesting that it and $\NIN$ qualify as ordinary mathematics.
This reference also establishes that Theorem \ref{cant2} is provable in a weak system involving only `computable' comprehension.  
By contrast, there are two `canonical' conservative extensions of second-order arithmetic $\Z_{2}$, called $\Z_{2}^{\omega}$ and $\Z_{2}^{\Omega}$, such that 
$\NIN$ \emph{cannot} be proved in $\Z_{2}^{\omega}$ and $\NIN$ can be proved in $\Z_{2}^{\Omega}$ (see Theorems \ref{dick} and \ref{flahu}).
Moreover, $\Z_{2}^{\omega}$ is based on \emph{third-order} functionals $\SS_{k}^{2}$ that can decide (second-order) $\Pi_{k}^{1}$-formulas, while $\Z_{2}^{\Omega}$ is based on Kleene's \emph{fourth-order} axiom $(\exists^{3})$.
We refer to Section \ref{prelim} for further details and definitions.  

\smallskip

Secondly, Turing's famous `machine' model introduced in \cite{tur37}, provides an intuitive and convincing formalism that captures the notion of `computing with real numbers'.  
This formalism does not apply to e.g.\ arbitrary $\R\di \R$-functions and Kleene later introduced his S1-S9 schemes which capture `computing with higher-order objects' (\cites{kleeneS1S9, longmann}).  
With this framework in mind, studying the computational properties of $\NIN$ means studying functionals $N$ satisfying the specification:
\be\tag{$\NIN(N)$}
(\forall Y:[0,1]\di \N)(  N(Y)(0)\ne_{\R} N(Y)(1)\wedge  Y(N(Y)(0))=Y(N(Y)(1)) ).
\ee
In a nutshell, $N(Y)=(x, y)$ computes the real numbers claimed to exist by $\NIN$.
As to precedent, the functional $N$ is a special case of Luckhardt's \emph{continuity indicators} from \cite{lucky}*{p.\ 243} which have the same functionality. 

\smallskip

Interpreting `computation' as in Kleene's S1-S9 (\cite{kleeneS1S9, longmann}), we show that $N$ as in $\NIN(N)$ \emph{cannot} be computed by any of the aforementioned `comprehension' functionals $\SS_{k}^{2}$ that give rise to $\Z_{2}^{\omega}$.
By contrast, the number $y$ in Theorem \ref{cant2} is outright (and efficiently) computable from the other data (\cite{grayk,simpson2}).  Our negative result is fundamentally based on a technique\footnote{Intuitively speaking, to build a model of $\Z_{2}^{\omega}+\neg\NIN$ or to show that $N$ as in $\NIN(N)$ is not (S1-S9) computable in any $\SS_{k}^{2}$, one starts with the observation that any $f\in 2^{\N}$ computable in some $\SS_{k}^{2}$, comes with some $e\in \N$, which is a code for the S1-S9-algorithm computing $f$ from $\SS_{k}^{2}$.  The Axiom of Choice of course provides a choice function $\Phi:2^{\N}\di \N$, i.e.\ $\Phi(f)=e$ with the previous notations, but Gandy selection (see Section \ref{xnx}) guarantees there is such a choice function $\Phi_{0}$ that is also S1-S9-computable relative to some $\SS_{k}^{2}$.  In this way, the type structure $\mathcal{M}$ consisting of all objects (S1-S9) computable in some $\SS_{k}^{2}$ has the desired properties:  $\mathcal{M}$ is trivially a model of $\Z_{2}^{\omega}$ and satisfies $\neg\NIN$, as $\Phi_{0}$ is (relative to $\mathcal{M}$) an injection from $2^{\N}$ to $\N$.} called \emph{Gandy selection} (see Section \ref{xnx}).    

\smallskip

In light of the above, the logical and computational properties of $\NIN$ \emph{expressed in terms of comprehension} are rather unsatisfactory.  
Indeed, the systems $\Z_{2}^{\omega}$ and $\Z_{2}^{\Omega}$ are both conservative extensions of $\Z_{2}$, but the former cannot prove $\NIN$ while the latter can. 
A similar phenomenon occurs for the computational properties of $\NIN$ as captured by the functional $N$ satisfying $\NIN(N)$ in Kleene's higher-order framework.  
It would be desirable to have a scale in which an intuitively\footnote{For instance, $\NIN$ does not imply any principle from the RM zoo (\cite{damirzoo}).  Our below results combined with \cite{kruisje}*{Theorem 3} show that $\NIN$ yields a conservative extension of arithmetical comprehension, as provided by Feferman's $\mu$-operator from Section \ref{lll}.} `weak' principle like $\NIN$ also falls into the formal `weak' category, and the same for the functional $N$ as in $\NIN(N)$.  The latter is \emph{strongly non-normal} following Definition~\ref{norma}.

\smallskip

The reason for this discrepancy is that \emph{we are comparing two fundamentally different categories}.  
Indeed, the functionals $\SS_{k}^{2}$ from $\Z_{2}^{\omega}$ and $\exists^{3}$ from $\Z_{2}^{\Omega}$ are `normal', i.e.\ they imply the existence of (and even compute) a \emph{discontinuous function} (say on $2^{\N}$).  By contrast, $\NIN$ and the functional $N$ from $\NIN(N)$ are `non-normal', implying that they \emph{do not} yield the existence of (let alone compute) discontinuous functions.   
It is an empirical observation (see \cites{dagsam, dagsamII, dagsamIII, dagsamV, dagsamVI, dagsamVII}) that measuring the strength of \emph{non-normal} objects and principles via \emph{normal} objects and principles always leads to the same unsatisfactory picture as in the previous paragraph based on $\Z_{2}^{\omega}$ and $\Z_{2}^{\Omega}$ in which intuitively `weak' principles are not assigned the formal `weak' category.  We provide a solution to all these problems in the next section. 

\subsubsection{The uncountability of $\R$ and the non-normal world}\label{nonnorm}
In this section, we sketch part of the non-normal world from \cite{samph}, and the place of $\NIN$ within it.  

\smallskip

First of all, the following figure provides an overview of some of our results for $\NIN$ and $\NBI$.   
Further definitions can be found in Section \ref{lll} while implications not involving $\NIN$ or $\NBI$ are in \cites{samph, dagsamIII, dagsamV, dagsamVII, dagsamVI}.
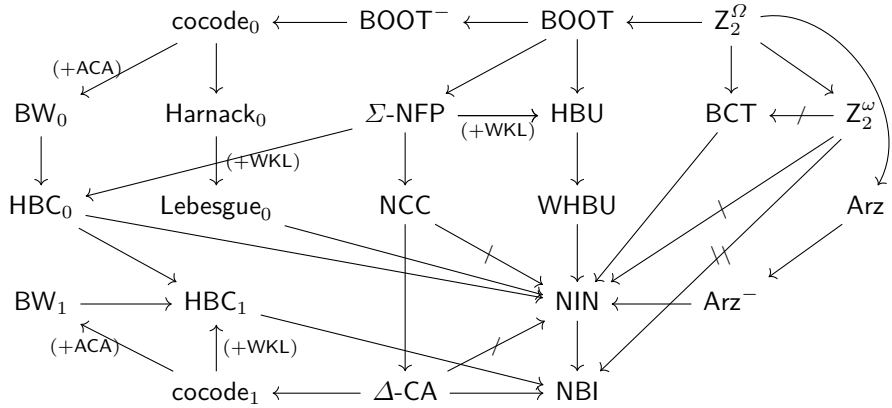
\begin{figure}[h]
\begin{tikzcd}
~& \arrow[ld, rightarrow,"(+\ACA)" anchor=east, bend right=0]\cocode_{0}\arrow[d]&\arrow[l, rightarrow,"" anchor=south, bend right=0]\BOOT^{-}& \arrow[l]\BOOT\arrow[d]\arrow[dl]&\arrow[l, bend right=0] \Z_{2}^{\Omega} \arrow{d} \arrow[rd] \arrow[rdd,rightarrow, bend left=65, anchor=north]&  \\
\BW_{0}\arrow[d]&\HAR_{0} \arrow[d, bend right=0]&\arrow[dll,rightarrow,"(+\WKL)" anchor=west]\Sigma\text{-}\NFP\arrow[d]  \arrow[r]\arrow[r,rightarrow,"(+\WKL)" anchor=north]&\HBU\arrow[d]&\BCT \arrow{ldd}& \Z_{2}^{\omega}\arrow[l,rightarrow,"/" anchor=center]\arrow[lldd,rightarrow,"\setminus" anchor=center]\arrow[llddd,rightarrow,"\setminus\setminus" anchor=center]~ \\
\HBC_{0}\arrow[rrrd]\arrow[rd]&\LEB_{0}\arrow[rrd]&  \NCC\arrow[dd, bend right=0]\arrow[dr,rightarrow,"/" anchor=center]  & \WHBU\arrow{d}&~ &  \arrow[ld, bend left=0]\Arz\\
\BW_{1}\arrow[r]&\HBC_{1}\arrow[rrd]& ~&\NIN \arrow[d]& \arrow[l]\Arz^{-} &  ~\\
~& \cocode_{1} \arrow[u, rightarrow,"(+\WKL)" anchor=west, bend right=0]\arrow[ul, rightarrow,"(+\ACA)" anchor=east, bend right=0]& \arrow[l]\DCA\arrow[ru,rightarrow,"/" anchor=center]\arrow[r]&\NBI & ~ &    
\end{tikzcd} 
\caption{Our main results in Reverse Mathematics}
\label{dd}
\end{figure}~\\
Secondly, we point out that $\Arz$ is Arzel\`a's convergence theorem \emph{for the Riemann integral}, published in 1885 (\cite{arse2}), i.e.\ ordinary mathematics if ever there was such.  
Curiously, $\Arz^{-}$, i.e.\ $\Arz$ formulated with \emph{Tao's metastability} in the conclusion, still implies $\NIN$.
Moreover, Figure \ref{dd} is only the tip of the proverbial iceberg: Appendix~\ref{extraextrareadallaboutit} contains more than a dozen basic theorems that imply $\NIN$ or $\NBI$.

\smallskip

Thirdly, we single out $\cocode_{i}$ from Figure \ref{dd} as it expresses that the word `countable' has the same meaning in RM and in mainstream mathematics, i.e.\ this `coding principle' seems crucial to anyone seeking to interpret the results of RM in a more general context.  For instance, $\textsf{Harnack}_{0}$ states \emph{a countable set has Lebesgue measure zero} (Harnack, 1885 \cite{harny}), while $\LEB_{0}$ is the \emph{Lebesgue criterion for Riemann integrability} restricted to countable sets.   
A similar observation can be made for the RM of topology (\cites{mummyphd,mummy,mummymf}), based as it is on \emph{countable} bases (see Example~\ref{mummyrem}).

\smallskip

Fourth, $\HBC_{i}$ captures the well-known Heine-Borel theorem for countable coverings, where `countable' has its \emph{usual/original} definition as used by Borel (see Sections~\ref{uncountsyn} and~\ref{uncountsym2}).
Curiously, these `countable' results are proved using $\HBU\di \NIN$, where the antecedent is the \emph{uncountable} Heine-Borel theorem.  Similarly, $\BW_{i}$ is the Bolzano-Weierstrass theorem providing suprema for countable sets in $[0,1]$. 
At the very least, these observations call for a thorough investigation of the role of `countable set' in RM.  This is all the more so in light of Theorem~\ref{BOOM} which shows that the Bolzano-Weierstrass theorem for \emph{countable sets}  in $2^{\N}$ yields $\SIX$ when combined with higher-order $\FIVE$, i.e.\ the Suslin functional.  As discussed in Remark \ref{pimp}, $\SIX$ seems to be the current `upper bound' of RM. 

\smallskip

Fifth, the negative results in Figure~\ref{dd} do not change if we add countable choice as in $\QFAC^{0,1}$ to $\Z_{2}^{\omega}$, except for the arrow that is crossed out twice.  
We stress that the functionals $\SS_{k}^{2}$ used to define $\Z_{2}^{\omega}$ are \emph{third-order} objects and that $\NIN$ and $\NBI$ are part of the language of \emph{third-order} arithmetic.  
By contrast, Kleene's $\exists^{3}$ used to define $\Z_{2}^{\Omega}$, is fundamentally \emph{fourth-order} in nature.  

\smallskip

Sixth, we discuss the non-normal nature of $\BOOT$, $\HBU$, and other principles.
To this end, consider the following implications, some of which are well-known.
\be\label{wrang}\tag{\textsf{P}}
\ACA_{0}\di \WKL_{0}\di \WWKL_{0}\di \RCA_{0} \textup{ and } \BOOT\di \HBU\di \WHBU \di \RCAo,
\ee
where we note that $\BOOT$ is an example of \emph{unconventional}\footnote{Formula classes like $\Pi_{k}^{1}$ allow for first- and second-order parameters (only), and the associated comprehension axiom is called \emph{conventional} comprehension (see e.g.\ Section \ref{lll}).  Now, $\BOOT$ as in Principle \ref{bopri} is formulated as $(\forall Y^{2})(\exists X \subset \N)(\forall n^{0})\big[ n\in X \asa (\exists f^{1})(Y(f, n)=0)    \big]$, i.e.\ comprehension involving third-order parameters $Y$.   To the best of our knowledge, \cites{littlefef, kohlenbach4} are the only places unconventional comprehension has been studied before.} comprehension. 

\smallskip

Recall that the $\ECF$-translation is the canonical embedding of higher-order into second-order arithmetic, as discussed in Remark \ref{ECF}.
The following crucial metamathematical properties \eqref{klo} and \eqref{klo2} about \eqref{wrang} and $\ECF$ are shown in \cite{samph}.
\begin{enumerate}
 \renewcommand{\theenumi}{\alph{enumi}}
\item The $\ECF$-translation maps the implications on the right of \eqref{wrang} to the implications on the left of \eqref{wrang}. \label{klo}
\item Under $\ECF$, equivalences to principles on the right of \eqref{wrang} are mapped to equivalences to principles on the left of \eqref{wrang}.  \label{klo2}
\end{enumerate}
Thus, item \eqref{klo} establishes that the right-hand side consists of non-normal principles, as $\ECF$ maps the existence of discontinuous functions to `$0=1$'.
An example of item~\eqref{klo2} is as follows: $\BOOT$ is equivalent to a certain monotone convergence theorem for \textbf{nets}; $\ECF$ translates this equivalence to the well-known equivalence
between $\ACA_{0}$ and the monotone convergence theorem for \textbf{sequences} (\cite{simpson2}*{III.2}). 

\smallskip

In light of \eqref{wrang} and items \eqref{klo} and \eqref{klo2}, the second-order world (involving $\ACA_{0}$ and weaker principles) is a reflection
of the non-normal world under $\ECF$.  
Similar results hold for $\ATR_{0}$ and $\FIVE$, as proved in \cite{samph}*{\S4}.   As expected, $\Z_{2}^{\omega}$ cannot prove $\BOOT$ as in \eqref{wrang}, but $\Z_{2}^{\Omega}$ can.   
The non-normal world as in \eqref{wrang} is (a small part of) a hierarchy based on the \emph{neighbourhood function principle} ($\NFP$; see \cite{samph}*{\S5} and Section \ref{bereft}), a classically valid continuity axiom 
from Brouwer's intuitionistic mathematics (\cite{troeleke1}). 
For the purposes of this paper, $\BOOT$ and \eqref{wrang} are sufficient. 

\smallskip

Finally, we believe the topic of this paper, namely what are the logical and computational properties of the uncountability of $\R$, to be of general interest to any audience of mathematicians as we identify surprising results about 
a very well-studied topic in the foundations of mathematics, namely the genesis of set theory.  
Beyond this, we have formulated the below proofs in such a way as to appeal to an as broad as possible audience.  In particular, the below is meant to showcase the techniques used to establish the results in \cites{dagsamIII, dagsamV, dagsamVI, dagsamVII}, which are part of our ongoing project on the logical and computational properties of the uncountable. 

\smallskip

Furthermore, by Figure \ref{dd}, $\NIN$ is implied by most of the (third-order) principles we have hitherto studied, e.g.\ the Lindel\"of lemma and the Heine-Borel $(\HBU)$, Vitali $(\WHBU)$, and Baire category ($\BCT$) theorems.  
Our results for $\NIN$ and $\NBI$, namely that they are provable in $\Z_{2}^{\Omega}$ and not in $\Z_{2}^{\omega}$ and extensions, imply that all stronger principles behave in the same way, thus reproving many results from \cite{dagsam, dagsamII, dagsamIII, dagsamV, dagsamVI ,dagsamVII}.

\section{Two frameworks}\label{prelim}
We discuss \emph{Reverse Mathematics} in Section \ref{prelim1} and introduce Kohlenbach's generalisation to \emph{higher-order arithmetic}, and the associated base theory $\RCAo$.  
We introduce higher-order \emph{computability theory}, following Kleene's computation schemes S1-S9, in Section \ref{HCT}. 
Based on this framework, we obtain some model constructions (Section \ref{models}) that are essential to our independence results in Section \ref{main}.  

\subsection{Reverse Mathematics}\label{prelim1}
We discuss Reverse Mathematics (Section \ref{introrm}) and introduce -in full detail- Kohlenbach's base theory of \emph{higher-order} Reverse Mathematics (Section \ref{rmbt}).
Some essential axioms, functionals, and notations may be found in Sections \ref{kkk} and \ref{lll}.
\subsubsection{Introduction}\label{introrm}
Reverse Mathematics (RM hereafter) is a program in the foundations of mathematics initiated around 1975 by Friedman (\cites{fried,fried2}) and developed extensively by Simpson (\cite{simpson2}).  
The aim of RM is to identify the minimal axioms needed to prove theorems of ordinary, i.e.\ non-set theoretical, mathematics. 

\smallskip

We refer to \cite{stillebron} for a basic introduction to RM and to \cite{simpson2, simpson1} for an overview of RM.  We expect basic familiarity with RM, but do sketch some aspects of Kohlenbach's \emph{higher-order} RM (\cite{kohlenbach2}) essential to this paper, including the base theory $\RCAo$ (Definition \ref{kase}).  

\smallskip

First of all, in contrast to `classical' RM based on \emph{second-order arithmetic} $\Z_{2}$, higher-order RM uses $\L_{\omega}$, the richer language of \emph{higher-order arithmetic}.  
Indeed, while the former is restricted to natural numbers and sets of natural numbers, higher-order arithmetic can accommodate sets of sets of natural numbers, sets of sets of sets of natural numbers, et cetera.  
To formalise this idea, we introduce the collection of \emph{all finite types} $\mathbf{T}$, defined by the two clauses:
\begin{center}
(i) $0\in \mathbf{T}$   and   (ii)  If $\sigma, \tau\in \mathbf{T}$ then $( \sigma \di \tau) \in \mathbf{T}$,
\end{center}
where $0$ is the type of natural numbers, and $\sigma\di \tau$ is the type of mappings from objects of type $\sigma$ to objects of type $\tau$.
In this way, $1\equiv 0\di 0$ is the type of functions from numbers to numbers, and  $n+1\equiv n\di 0$.  Viewing sets as given by characteristic functions, we note that $\Z_{2}$ only includes objects of type $0$ and $1$.    

\smallskip

Secondly, the language $\L_{\omega}$ includes variables $x^{\rho}, y^{\rho}, z^{\rho},\dots$ of any finite type $\rho\in \mathbf{T}$.  Types may be omitted when they can be inferred from context.  
The constants of $\L_{\omega}$ include the type $0$ objects $0, 1$ and $ <_{0}, +_{0}, \times_{0},=_{0}$  which are intended to have their usual meaning as operations on $\N$.
Equality at higher types is defined in terms of `$=_{0}$' as follows: for any objects $x^{\tau}, y^{\tau}$, we have
\be\label{aparth}
[x=_{\tau}y] \equiv (\forall z_{1}^{\tau_{1}}\dots z_{k}^{\tau_{k}})[xz_{1}\dots z_{k}=_{0}yz_{1}\dots z_{k}],
\ee
if the type $\tau$ is composed as $\tau\equiv(\tau_{1}\di \dots\di \tau_{k}\di 0)$.  
Furthermore, $\L_{\omega}$ also includes the \emph{recursor constant} $\mathbf{R}_{\sigma}$ for any $\sigma\in \mathbf{T}$, which allows for iteration on type $\sigma$-objects as in the special case \eqref{special}.  Formulas and terms are defined as usual.  
One obtains the sub-language $\L_{n+2}$ by restricting the above type formation rule to produce only type $n+1$ objects (and related types of similar complexity).        

\subsubsection{The base theory of higher-order Reverse Mathematics}\label{rmbt}
We introduce Kohlenbach's base theory $\RCAo$, first introduced in \cite{kohlenbach2}*{\S2}.
\bdefi\label{kase} 
The base theory $\RCAo$ consists of the following axioms.
\begin{enumerate}
 \renewcommand{\theenumi}{\alph{enumi}}
\item  Basic axioms expressing that $0, 1, <_{0}, +_{0}, \times_{0}$ form an ordered semi-ring with equality $=_{0}$.
\item Basic axioms defining the well-known $\Pi$ and $\Sigma$ combinators (aka $K$ and $S$ in \cite{avi2}), which allow for the definition of \emph{$\lambda$-abstraction}. 
\item The defining axiom of the recursor constant $\mathbf{R}_{0}$: for $m^{0}$ and $f^{1}$: 
\be\label{special}
\mathbf{R}_{0}(f, m, 0):= m \textup{ and } \mathbf{R}_{0}(f, m, n+1):= f(n, \mathbf{R}_{0}(f, m, n)).
\ee
\item The \emph{axiom of extensionality}: for all $\rho, \tau\in \mathbf{T}$, we have:
\be\label{EXT}\tag{$\textsf{\textup{E}}_{\rho, \tau}$}  
(\forall  x^{\rho},y^{\rho}, \varphi^{\rho\di \tau}) \big[x=_{\rho} y \di \varphi(x)=_{\tau}\varphi(y)   \big].
\ee 
\item The induction axiom for quantifier-free formulas of $\L_{\omega}$.
\item $\QFAC^{1,0}$: the quantifier-free Axiom of Choice as in Definition \ref{QFAC}.
\end{enumerate}
\edefi
\noindent
Note that variables (of any finite type) are allowed in quantifier-free formulas of the language $\L_{\omega}$: only quantifiers are banned.
Recursion as in \eqref{special} is called \emph{primitive recursion}; the class of functionals obtained from $\mathbf{R}_{\rho}$ for all $\rho \in \mathbf{T}$ is called \emph{G\"odel's system $T$} of all (higher-order) primitive recursive functionals. 
\bdefi\label{QFAC} The axiom $\QFAC$ consists of the following for all $\sigma, \tau \in \textbf{T}$:
\be\tag{$\QFAC^{\sigma,\tau}$}
(\forall x^{\sigma})(\exists y^{\tau})A(x, y)\di (\exists Y^{\sigma\di \tau})(\forall x^{\sigma})A(x, Y(x)),
\ee
for any quantifier-free formula $A$ in the language of $\L_{\omega}$.
\edefi
As discussed in \cite{kohlenbach2}*{\S2}, $\RCAo$ and $\RCA_{0}$ prove the same sentences `up to language' as the latter is set-based and the former function-based.   
This conservation results is obtained via the so-called $\ECF$-interpretation, which we now discuss. 
\begin{rem}[The $\ECF$-interpretation]\label{ECF}\rm
The (rather) technical definition of $\ECF$ may be found in \cite{troelstra1}*{p.\ 138, \S2.6}.
Intuitively, the $\ECF$-interpretation $[A]_{\ECF}$ of a formula $A\in \L_{\omega}$ is just $A$ with all variables 
of type two and higher replaced by type one variables ranging over so-called `associates' or `RM-codes' (see \cite{kohlenbach4}*{\S4}); the latter are (countable) representations of continuous functionals.  
The $\ECF$-interpretation connects $\RCAo$ and $\RCA_{0}$ (see \cite{kohlenbach2}*{Prop.\ 3.1}) in that if $\RCAo$ proves $A$, then $\RCA_{0}$ proves $[A]_{\ECF}$, again `up to language', as $\RCA_{0}$ is 
formulated using sets, and $[A]_{\ECF}$ is formulated using types, i.e.\ using type zero and one objects.  
\end{rem}
In light of the widespread use of codes in RM and the common practise of identifying codes with the objects being coded, it is no exaggeration to refer to $\ECF$ as the \emph{canonical} embedding of higher-order into second-order arithmetic. 

\smallskip

Finally as noted above, Theorem \ref{cant2} is provable in the base theory. 
\begin{thm}[\cite{simpson2}*{II.4.9}] The following is provable in $\RCA_{0}$.
For any sequence of real numbers $(x_{n})_{n\in \N}$, there is a real $y$ different from $x_{n}$ for all $n\in \N$.
\end{thm}

\subsubsection{Notations and the like}\label{kkk}
We introduce the usual notations for common mathematical notions, like real numbers, as also introduced in \cite{kohlenbach2}.  
\begin{defi}[Real numbers and related notions in $\RCAo$]\label{keepintireal}\rm~
\begin{enumerate}
 \renewcommand{\theenumi}{\alph{enumi}}
\item Natural numbers correspond to type zero objects, and we use `$n^{0}$' and `$n\in \N$' interchangeably.  Rational numbers are defined as signed quotients of natural numbers, and `$q\in \Q$' and `$<_{\Q}$' have their usual meaning.    
\item Real numbers are coded by fast-converging Cauchy sequences $q_{(\cdot)}:\N\di \Q$, i.e.\  such that $(\forall n^{0}, i^{0})(|q_{n}-q_{n+i}|<_{\Q} \frac{1}{2^{n}})$.  
We use Kohlenbach's `hat function' from \cite{kohlenbach2}*{p.\ 289} to guarantee that every $q^{1}$ defines a real number.  
\item We write `$x\in \R$' to express that $x^{1}:=(q^{1}_{(\cdot)})$ represents a real as in the previous item and write $[x](k):=q_{k}$ for the $k$-th approximation of $x$.    
\item Two reals $x, y$ represented by $q_{(\cdot)}$ and $r_{(\cdot)}$ are \emph{equal}, denoted $x=_{\R}y$, if $(\forall n^{0})(|q_{n}-r_{n}|\leq {2^{-n+1}})$. Inequality `$<_{\R}$' is defined similarly.  
We sometimes omit the subscript `$\R$' if it is clear from context.           
\item Functions $F:\R\di \R$ are represented by $\Phi^{1\di 1}$ mapping equal reals to equal reals, i.e.\ extensionality as in $(\forall x , y\in \R)(x=_{\R}y\di \Phi(x)=_{\R}\Phi(y))$.\label{EXTEN}
\item The relation `$x\leq_{\tau}y$' is defined as in \eqref{aparth} but with `$\leq_{0}$' instead of `$=_{0}$'.  Binary sequences are denoted `$f^{1}, g^{1}\leq_{1}1$', but also `$f,g\in C$' or `$f, g\in 2^{\N}$'.  Elements of Baire space are given by $f^{1}, g^{1}$, but also denoted `$f, g\in \N^{\N}$'.
\item For a binary sequence $f^{1}$, the associated real in $[0,1]$ is $\r(f):=\sum_{n=0}^{\infty}\frac{f(n)}{2^{n+1}}$.\label{detrippe}
\item Sets of type $\rho$ objects $X^{\rho\di 0}, Y^{\rho\di 0}, \dots$ are given by their characteristic functions $F^{\rho\di 0}_{X}\leq_{\rho\di 0}1$, i.e.\ we write `$x\in X$' for $ F_{X}(x)=_{0}1$. \label{koer} 
\end{enumerate}
\end{defi}
For completeness, we list the following notational convention for finite sequences.  
\begin{nota}[Finite sequences]\label{skim}\rm
The type for `finite sequences of objects of type $\rho$' is denoted $\rho^{*}$, which we shall only use for $\rho=0,1$.  
Since the usual coding of pairs of numbers goes through in $\RCAo$, we shall not always distinguish between $0$ and $0^{*}$. 
Similarly, we assume a fixed coding for finite sequences of type $1$ and shall make use of the type `$1^{*}$'.  
In general, we do not always distinguish between `$s^{\rho}$' and `$\langle s^{\rho}\rangle$', where the former is `the object $s$ of type $\rho$', and the latter is `the sequence of type $\rho^{*}$ with only element $s^{\rho}$'.  The empty sequence for the type $\rho^{*}$ is denoted by `$\langle \rangle_{\rho}$', usually with the typing omitted.  

\smallskip

Furthermore, we denote by `$|s|=n$' the length of the finite sequence $s^{\rho^{*}}=\langle s_{0}^{\rho},s_{1}^{\rho},\dots,s_{n-1}^{\rho}\rangle$, where $|\langle\rangle|=0$, i.e.\ the empty sequence has length zero.
For sequences $s^{\rho^{*}}, t^{\rho^{*}}$, we denote by `$s*t$' the concatenation of $s$ and $t$, i.e.\ $(s*t)(i)=s(i)$ for $i<|s|$ and $(s*t)(j)=t(|s|-j)$ for $|s|\leq j< |s|+|t|$. For a sequence $s^{\rho^{*}}$, we define $\overline{s}N:=\langle s(0), s(1), \dots,  s(N-1)\rangle $ for $N^{0}<|s|$.  
For a sequence $\alpha^{0\di \rho}$, we also write $\overline{\alpha}N=\langle \alpha(0), \alpha(1),\dots, \alpha(N-1)\rangle$ for \emph{any} $N^{0}$.  By way of shorthand, 
$(\forall q^{\rho}\in Q^{\rho^{*}})A(q)$ abbreviates $(\forall i^{0}<|Q|)A(Q(i))$, which is (equivalent to) quantifier-free if $A$ is.   
\end{nota}
\subsubsection{Some axioms and functionals}\label{lll}
As noted in Section \ref{intro}, the logical hardness of a theorem is measured via what fragment of the comprehension axiom is needed for a proof.  
For this reason, we introduce some axioms and functionals related to \emph{higher-order comprehension} in this section.
We are mostly dealing with \emph{conventional} comprehension here, i.e.\ only parameters over $\N$ and $\N^{\N}$ are allowed in formula classes like $\Pi_{k}^{1}$ and $\Sigma_{k}^{1}$.

\smallskip

First of all, the following functional is clearly discontinuous at $f=11\dots$; in fact, $(\exists^{2})$ is equivalent to the existence of $F:\R\di\R$ such that $F(x)=1$ if $x>_{\R}0$, and $0$ otherwise (\cite{kohlenbach2}*{\S3}).  This fact shall be repeated often.  
\be\label{muk}\tag{$\exists^{2}$}
(\exists \varphi^{2}\leq_{2}1)(\forall f^{1})\big[(\exists n)(f(n)=0) \asa \varphi(f)=0    \big]. 
\ee
Related to $(\exists^{2})$, the functional $\mu^{2}$ in $(\mu^{2})$ is also called \emph{Feferman's $\mu$} (\cite{avi2}).
\begin{align}\label{mu}\tag{$\mu^{2}$}
(\exists \mu^{2})(\forall f^{1})\big[ (\exists n)(f(n)=0) \di [f(\mu(f))=0&\wedge (\forall i<\mu(f))(f(i)\ne 0) ]\\
& \wedge [ (\forall n)(f(n)\ne0)\di   \mu(f)=0]    \big], \notag
\end{align}
We have $(\exists^{2})\asa (\mu^{2})$ over $\RCAo$ and $\ACAo\equiv\RCAo+(\exists^{2})$ proves the same sentences as $\ACA_{0}$ by \cite{hunterphd}*{Theorem~2.5}. 

\smallskip

Secondly, the functional $\SS^{2}$ in $(\SS^{2})$ is called \emph{the Suslin functional} (\cite{kohlenbach2}).
\be\tag{$\SS^{2}$}
(\exists\SS^{2}\leq_{2}1)(\forall f^{1})\big[  (\exists g^{1})(\forall n^{0})(f(\overline{g}n)=0)\asa \SS(f)=0  \big], 
\ee
The system $\FIVE^{\omega}\equiv \RCAo+(\SS^{2})$ proves the same $\Pi_{3}^{1}$-sentences as $\FIVE$ by \cite{yamayamaharehare}*{Theorem 2.2}.   
By definition, the Suslin functional $\SS^{2}$ can decide whether a $\Sigma_{1}^{1}$-formula as in the left-hand side of $(\SS^{2})$ is true or false.   We similarly define the functional $\SS_{k}^{2}$ which decides the truth or falsity of $\Sigma_{k}^{1}$-formulas from $\L_{2}$; we also define 
the system $\SIXK$ as $\RCAo+(\SS_{k}^{2})$, where  $(\SS_{k}^{2})$ expresses that $\SS_{k}^{2}$ exists.  
We note that the operators $\nu_{n}$ from \cite{boekskeopendoen}*{p.\ 129} are essentially $\SS_{n}^{2}$ strengthened to return a witness (if existant) to the $\Sigma_{n}^{1}$-formula at hand.  

\smallskip

\noindent
Thirdly, full second-order arithmetic $\Z_{2}$ is readily derived from $\cup_{k}\SIXK$, or from:
\be\tag{$\exists^{3}$}
(\exists E^{3}\leq_{3}1)(\forall Y^{2})\big[  (\exists f^{1})(Y(f)=0)\asa E(Y)=0  \big], 
\ee
and we therefore define $\Z_{2}^{\Omega}\equiv \RCAo+(\exists^{3})$ and $\Z_{2}^\omega\equiv \cup_{k}\SIXK$, which are conservative over $\Z_{2}$ by \cite{hunterphd}*{Cor.\ 2.6}. 
Despite this close connection, $\Z_{2}^{\omega}$ and $\Z_{2}^{\Omega}$ can behave quite differently, as discussed in e.g.\ \cite{dagsamIII}*{\S2.2}.   
The functional from $(\exists^{3})$ is also called `$\exists^{3}$', and we use the same convention for other functionals.  

\smallskip

Fourth, the Heine-Borel theorem states the existence of a finite sub-covering for an open covering of certain spaces. 
Now, a functional $\Psi:\R\di \R^{+}$ gives rise to a \emph{canonical cover} $\cup_{x\in I} I_{x}^{\Psi}$ for $I\equiv [0,1]$, where $I_{x}^{\Psi}$ is the open interval $(x-\Psi(x), x+\Psi(x))$.  
Hence, the uncountable covering $\cup_{x\in I} I_{x}^{\Psi}$ has a finite sub-covering by the Heine-Borel theorem; in symbols:
\begin{princ}[$\HBU$]
$(\forall \Psi:\R\di \R^{+})(\exists  y_{0}, \dots, y_{k}\in I){(\forall x\in I)}(\exists i\leq k)(x\in I_{y_{i}}^{\Psi}).$
\end{princ}
Note that $\HBU$ is almost verbatim \emph{Cousin's lemma} (see \cite{cousin1}*{p.\ 22}), i.e.\ the Heine-Borel theorem restricted to canonical covers.  
This restriction does not make a difference, as studied in \cite{sahotop}.
Let $\WHBU$ be the following weakening of $\HBU$:
\begin{princ}[$\WHBU$] For any $\Psi:\R\di \R^{+} $ and $ \eps>_{\R}0$, there are pairwise distinct $ y_{0}, \dots, y_{k}\in I$ with $1-\eps <_{\R}\sum_{i\leq k}|J_{y_{i}}^{\Psi}| $, where $J_{y_{i+1}}^{\Psi}:= I_{y_{i+1}}^{\Psi}\setminus (\cup_{j\leq i}I_{y_{i}}^{\Psi}) $.
\end{princ}
Note that $\WHBU$ expresses the essence of the Vitali covering theorem for \emph{uncountable} coverings; Vitali already
considered the latter in \cite{vitaliorg}.  
By \cite{dagsamIII, dagsamV,dagsamVI}, $\Z_{2}^{\Omega}$ proves $\HBU$ and $\WHBU$ but $\Z_{2}^{\omega}+\QFAC^{0,1}$ cannot.
Basic properties of the \emph{gauge integral} (\cite{zwette, mullingitover}) are equivalent to $\HBU$ while $\WHBU$ is equivalent to basic properties of the Lebesgue integral (without RM-codes; \cite{dagsamVI}).  

\smallskip

We note that $\HBU$ (resp.\ $\WHBU$) is the higher-order counterpart of $\WKL$ (resp.\ $\WWKL$), i.e.\ \emph{weak K\"onig's lemma} (resp.\ \emph{weak weak K\"onig's lemma}) from \cite{simpson2}*{IV and X} as $\ECF$ maps $\HBU$ (resp.\ $\WHBU$) to $\WKL$ (resp.\ $\WWKL$), i.e.\ these are (intuitively) \emph{weak} principles. 

\smallskip

Finally, the aforementioned results suggest that (higher-order) comprehension as in $\SIXK$ is not the right way of measuring the strength of $\HBU$. 
As a better alternative, we have introduced the following axiom in \cite{samph}.
\begin{princ}[$\BOOT$]\label{bopri}
$(\forall Y^{2})(\exists X \subset \N)(\forall n^{0})\big[ n\in X \asa (\exists f^{1})(Y(f, n)=0)    \big]. $
\end{princ}
By \cite{samph}*{\S3}, $\BOOT$ is equivalent to convergence theorems for \emph{nets}, we have the implication $\BOOT\di \HBU$, and $\RCAo+\BOOT$ has the same first-order strength as $\ACA_{0}$.  
Moreover, $\BOOT$ is a natural fragment of Feferman's \emph{projection axiom} $\textsf{(Proj1)}$ from \cite{littlefef}.
Thus, $\BOOT$ is a natural axiom that provides a better `scale' for measuring the strength of $\HBU$ and its ilk, as discussed in \cite{samph} and Section \ref{nonnorm}.

\subsection{Higher-order computability theory}\label{HCT}
\subsubsection{Introduction}\label{floep2}
As noted above, some of our main results will be proved using techniques from computability theory.
Thus, we first make our notion of `computability' precise as follows.  
\begin{enumerate}
\item[(I)] We adopt $\ZFC$, i.e.\ Zermelo-Fraenkel set theory with the Axiom of Choice, as the official metatheory for all results, unless explicitly stated otherwise.
\item[(II)] We adopt Kleene's notion of \emph{higher-order computation} as given by his nine clauses S1-S9 (see \cite{longmann}*{Ch.\ 5} or \cite{kleeneS1S9}) as our official notion of `computable'.
\end{enumerate}
Similar to \cites{dagsam,dagsamII, dagsamIII, dagsamV, dagsamVI, dagsamVII}, one main aim of this paper is the study of functionals of type 3 that are \emph{natural} from the perspective of mathematical practise. 

\smallskip

We refer to \cite{longmann} for a thorough and recent overview of higher-order computability theory.  
We provide an intuitive introduction to S1-S9 in Sections \ref{esnaain} and \ref{kasic}.
We also sketch one of our main techniques, called \emph{Gandy selection}, in Section \ref{xnx}.  
Intuitively speaking, this method expresses that S1-S9 computability satisfies an effective version of the Axiom of Choice. 
Using this technique, we construct a number of models in Section \ref{models} that establish our independence results in Section~\ref{main}.

\smallskip

Finally, we mention the distinction between `normal' and `non-normal' functionals  based on the following definition from \cite{longmann}*{\S5.4}. 
We note that $\exists^{n}$ is essentially just $\exists^{3}$ from Section \ref{lll} with all types `bumped up' to level $n$.  
\begin{defi}\label{norma}
For $n\geq 2$, a functional of type $n$ is called \emph{normal} if it computes Kleene's $\exists^{n}$ following S1-S9, and \emph{non-normal} otherwise.  
\end{defi}
Similarly, we call a statement about type $n$ objects ($n\geq 2$) \emph{normal} if it implies the existence of $\exists^{n}$ over Kohlenbach's base theory from Section \ref{prelim1}, and \emph{non-normal} otherwise.  
We also use `\emph{strongly} non-normal' for type $3$ functionals that do not compute $\exists^{3}$ \emph{relative to $\exists^{2}$}.  
The realiser $N$ of $\NIN$ from Section \ref{norm} is a natural example of a non-normal functional, as discussed in Section \ref{floep}.

\smallskip

In this paper, all principles we study are part of third-order arithmetic, i.e.\ `non-normal vs normal' refers to $\exists^{2}$ in this case.  
The associated realisers are fourth-order, i.e.\ `non-normal vs normal' then refers to $\exists^{3}$.  
Note that by \cite{kohlenbach2}*{\S3}, $(\exists^{2})$ is equivalent to the existence of a discontinuous function on $\R$.

\subsubsection{Kleene's computation schemes}\label{esnaain}
For those familiar with \emph{Turing} computability, Kleene's S1-S9 in a nutshell is as follows: the schemes S1-S8 merely introduce (higher-order) primitive recursion, while S9 essentially states that the
recursion theorem holds. In this section we will provide a slightly more detailed introduction to Kleene computability.  All further details can be found in  \cite{longmann}*{\S 5.1.4}.
\begin{definition}{\em A \emph{type structure} ${ \Tp}$ is a sequence $\{\Tp[k]\}_{k \in \N}$ as follows.
\begin{itemize} 
\item $\Tp[0] = \N$.
\item For all $k \in \N$, $\Tp[k+1]$ is a set of functions $\Phi:\Tp[k] \rightarrow \N$.
\end{itemize}}
\end{definition} 
We note that $\Tp$ involves only total objects.  
As we will see, the Kleene schemes can be interpreted for all type structures. One of our applications of   type structures is that they will serve as models for  fragments of higher-order arithmetic, structures for the language $\L_{\omega}$. While the Kleene schemes are defined for pure types, the language $\L_{\omega}$ is over a richer set of types, known as the \emph{finite} types. However, assuming some modest closure properties of a type structure $ \Tp$, the extension to the finite types is unique (see \cite{longmann}*{\S4.2}). This is the case when $\Tp$ is \emph{Kleene closed} as in Definition \ref{godalmach}.

\smallskip

The following main definition is  \cite{longmann}*{Definition 5.1.1} adjusted to a type structure $\Tp$.
We assume a standard sequence numbering over $\N$: variables $a$, $b$, $c$, $x$, $y$, $q$, $e$, and $d$ denote elements of $\N$, while $f$, $g, \dots$ denote elements of $\Tp[1]$. 
We let $\Phi_i^{k_i}$ stand for an element of $\Tp[k_i]$.
The index `$e$' in \eqref{K} serves as a G\"odel number denoting the $e$-th Kleene algorithm. 
We use the notation `$\{e\}_\TP$' if we need to specify the particular type structure $\Tp$.
\begin{definition}[Kleene S1-S9 schemes relative to $\TP$]
{\em Let $\TP = \{\Tp[k]\}_{k \in \N}$ be a type structure. Over the latter, we define the relation
\be\label{K}\tag{\textsf{K}}
\{e\}(\Phi_0^{k_0} , \ldots , \Phi_{n-1}^{k_{n-1}}) = a
\ee
by a monotone inductive definition as follows. We omit the upper indices for the types whenever they are clear from context.
\begin{itemize}
\item[S1] If $e = \langle 1 \rangle$, then $\{e\}(x,\vec \Phi) = x+1$.
\item[S2] If $e = \langle 2,q\rangle$, then $\{e\}(\vec \Phi) = q$.
\item[S3] If $e = \langle 3 \rangle$, then $ \{e\}(x,\vec \Phi) = x$.
\item[S4] If $e = \langle 4 , e_0 , e_1\rangle$, then $\{e\}(\vec \Phi) = a$ if for some $b$ we have that $\{e_1\}(\vec \Phi) = b$ and $\{e_0\}(b,\vec \Phi) = a$.
\item[S5] If $e = \langle 5 , e_0 , e_1 \rangle$ then
\begin{itemize}
\item[-] $\{e\}(0, \vec \Phi) = a$ if $\{e_0\}(\vec \Phi) = a$,
\item[-] $\{e\}(x+1 , \vec \Phi) = a$ if there is some $b$ such that $\{e\}(x,\vec \Phi) = b$ and $\{e_1\}(b, \vec \Phi) = a$.

\end{itemize}
\item[S6] If $e = \langle 6 , d , \pi(0) , \ldots , \pi(n-1)\rangle$, where $\pi$ is a permutation of $\{0 , \ldots , n-1\}$, then $\{e\}(\Phi_0 , \ldots , \Phi_{n-1}) = a$ if $\{d\}(\Phi_{\pi(0)} , \ldots , \Phi_{\pi(n-1)}) = a$.
\item[S7] If $e = \langle 7 \rangle$, then $\{e\}(f,x,\vec \Phi) = f(x)$.
\item[S8] If $e = \langle 8, d\rangle$, then $\{e\}(\Phi^{k+2}, \vec \Phi) = a$ if there is a $\phi^{k+1} \in \Tp[k+1]$ such that $\Phi(\phi) = a$ and for all $\xi^k \in \Tp[k]$ we have that   $\{d\}(\Phi , \xi , \vec \Phi) = \phi(\xi)$.
\item[S9] If $e = \langle 9,m+1\rangle$ and $m \leq n$, $\{e\}(d,\Phi_0 , \ldots , \Phi_{n}) = a$ if $\{d\}(\Phi_0 , \ldots , \Phi_{m}) = a$.
\end{itemize}}
\end{definition}
Intuitively speaking, $\{e\}(\vec \Phi) = a$ represents the result of a terminating computation, where S1, S2, S3, and S7 provide us with the initial computation, S4 is composition, S5 represents primitive recursion, and S6 represents permutations of arguments. Finally, S8 represents higher-order-composition \emph{with two requirements}: that the computation of $\{d\}(\Phi , \xi , \vec \Phi)$ must terminate for all $\xi \in \Tp[k]$ and that the functional $\phi$ thus computed must be an element of $\Tp[k+1]$. In the original definition from \cite{kleeneS1S9}, the schemes are interpreted over the maximal type structure and the latter requirement is vacuous.  This requirement is essential to our development, as will become clear below. 
\begin{definition}[Kleene computability]\label{uneek}
{\em  ~ \begin{itemize}\label{godalmach}
\item[(a)] Let $\Tp$ be a type structure, let $\phi:\Tp[k] \rightarrow \N$, and let $\vec \Phi$ be in $\TP$ as above.  We say that $\phi$ is \emph{Kleene computable} in $\vec \Phi$ (over $\TP$) if there is an index $e$ such that for all $\xi \in \Tp[k]$ we have that $\{e\}(\xi , \vec \Phi) = \phi(\xi)$.
\item[(b)] The type structure $\TP$ is \emph{Kleene closed} if for all $k$ and all $\phi:\Tp[k] \rightarrow \N$ that are Kleene computable in elements in $\TP$, we have that $\phi \in \Tp[k+1]$. 
\end{itemize}}
\end{definition}
When a type structure $\TP$ is Kleene closed, it will have a canonical extension to an interpretation $\Tp[\sigma]$ for all finite types $\sigma$ as in the language $\L_\omega$. This is folklore and is discussed at length in \cite[\S 4.2]{longmann}. 
We use $\TP^*$ to denote this unique extension. 
What is important to us is that if $\TP$ is Kleene closed, then $\TP^*$ is a model of $\RCA_0^\omega$ and all terms in G\"odel's $T$ have canonical interpretations in $\TP^*$.

\smallskip

Finally, we motivate our choice of framework as follows.  
\begin{remark}[Church-Turing-Kleene]{\em 
First of all, there is no `Church-Turing thesis' for higher-order computability theory: there are several competing concepts based on $\mu$-computability, fixed point constructors, partial functionals, and so forth.  
We refer to \cite{longmann} for a detailed overview. We have primarily used Kleene's concept because it has proved to be useful in the construction of structures for extensions of $\RCA_0^\omega$, proving theorems of logical independence. 

\smallskip

Secondly, as a kind of `dichotomy phenomenon', we have observed that when we prove that one object of interest is computable in another object of interest, we can usually do this within G\"odel's $T$, or even with a weaker notion of relative computability.  By contrast, when we prove a \emph{non-computability} result, we can do this for the (relatively strong) concept of Kleene computability, but \emph{not necessarily} for stronger concepts.}
\end{remark}
In the next sections, we discuss some aspects of Kleene computability essential to this paper while foregoing a full introduction.  
For instance, we will not prove that \emph{being computable in} is transitive, while we freely use this fact.  

\subsubsection{Basic results in Kleene computability theory}\label{kasic}
We discuss two folklore results (Lemmas \ref{lemma.absolute} and \ref{lemma.blowout}) and one important theorem (Theorem \ref{thm.ext}).

\smallskip

As to the first folklore result, since the relation $\{e\}(\vec \Phi) = a$ is defined by a positive inductive definition, all such \emph{computation tuples} $\langle e , \vec \Phi , a \rangle$ will either stem from one of the schemes S1, S2, S3 or S7, or there will be a unique base of other such computation tuples. The elements of this base are called \emph{the immediate sub-computations} and the transitive closure of this relation gives us the well-founded relation of \emph{being a sub-computation}. 
The ordinal rank of a terminating computation is, by definition, the rank of the corresponding computation tuple in this well-founded relation.
The list of arguments in a sub-computation may contain more objects, but as they are only added through the use of S8, they will only be two type levels (or more) below the maximal type in the original list. In particular this means that if all arguments in a computation $\{e\}(\vec \Phi) = a$ are of type 2 or lower, all extra arguments in the sub-computations will be integers. This further means that in a computation like this, say of the form $\{e\}(\vec F , \vec f , \vec a) = b$, where each $F_i$ is of type 2, whenever we use the value  $F_i(g)$ in a sub-computation based on S8, $g$ itself will be computable in the argument list.  
For the sake of notational simplicity, we often order the arguments of a computation according to their types.  
\begin{lemma}[Restriction] \label{lemma.absolute}
 Let $A \subseteq  \N^\N$ and consider $F:\N^\N \rightarrow \N$ with its restriction to $A$ denoted $F_A$. If all $f$ computable in $F$ and elements in $A$ are in $A$, then for all indices $e$, for all $\vec f$ from $A$ and all $\vec a$ we have that \[\{e\}(F,\vec f , \vec a) \simeq \{e\}(F_A , \vec f , \vec a)\] where `$\simeq$' is the \emph{Kleene-equality}: both sides are either undefined or both sides are defined and equal.
 \end{lemma}
We also have the following lemma, combining the observation of Lemma \ref{lemma.absolute} with the transitivity of the relation \emph{computable in}.
\begin{lemma}[Sandwich]\label{lemma.blowout} 
Let $F,G:\N^\N \rightarrow \N$ and assume that there is a partial functional $H \subseteq G$ that is partially computable in $F$ and total on the set of functions computable in $F$. 
Then all functions computable in $G$ are computable in $F$. 
\end{lemma}
As to the second folklore result, the following theorem shows that we have a high degree of flexibility when defining type structures from sets of functionals. 
\begin{theorem}\label{thm.ext} 
Let $A \subseteq \N^\N$ and let $B$ be a set of functionals $F:A \rightarrow \N$. 
Assume that all $f$ computable in a sequence from $B$ and $A$  are in $A$. 
Then there is a Kleene closed type structure $\TP$ such that $A = \Tp[1]$ and $B \subseteq \Tp[2]$.
\end{theorem}
\begin{proof} We define $\Tp[k]$ by recursion on $k$ as follows:
\begin{itemize}
\item $\Tp[0] = \N$
\item $\Tp[k+1]$ is the set of functionals $\Phi:\Tp[k] \rightarrow \N$ that are computable over $\{\Tp[i]\}_{i \leq k}$ in a finite sequence from $A$ and $B$. 
\end{itemize}
Since sub-computations only involve extra arguments of lower types, computability over $\{\Tp[i]\}_{i \leq k}$ will be the same as computability over $\TP$, and since \emph{computable in} is transitive, $\TP$ will be Kleene closed. 
\end{proof}
The \emph{fan functional} is a non-normal object computing a modulus of uniform continuity for continuous inputs (\cite{longmann}*{\S8.3}).  

\begin{corollary} 
There is a Kleene closed type structure $\TP$ such that $\Tp[1]$ satisfies $\ACA_0$, but with no fan functional in $\Tp[3]$.
\end{corollary}
\begin{proof} Let $A$ be the set of arithmetically defined functions and $B$ be empty. Since the fan functional is not Kleene computable in any function (\cite{longmann}*{\S8.3}), the type structure derived from the proof of Theorem \ref{thm.ext} will have this property. 
\end{proof}
All elements in $\Tp[2]$ in the proof are continuous with an associate in $\Tp[1]$.  
Moreover, there is a Kleene closed type structure $\TP'$ containing all of $\N^\N$, but not the fan functional, by the same argument.
In this case, $\Tp'[2]$ are exactly the continuous functionals, implying that the fan functional is total, but not in $\Tp'[3]$.

\subsubsection{Advanced Kleene computability theory}\label{xnx}
We discuss \emph{Gandy selection}, a technique expressing that Kleene computability satsifies a kind of `computable Axiom of Choice'.  
This technique is based on \emph{stage comparison}, discussed next. 

\smallskip

Now, given a type structure $\TP$, Kleene computations are defined by a positive inductive definition and the set of computations $\{e\}(\vec \Phi) = a$ are equipped with \emph{ordinal ranks}.  For the latter, we refer to \cite{longmann}*{\S5.1.1} and the first folklore result in Section \ref{kasic}.
Moreover, if the elements of $\vec \Phi$ have types $\leq 2$, this ordinal rank is \emph{countable}, which holds for any Kleene closed type structure. 
\begin{defi}[Norm of a computation]
If $\{e\}(F , \vec f , \vec a)$ terminates, we let the ordinal $||\langle e, F , \vec f , \vec a\rangle||$ denote the norm. If the computation does not terminate, we set the norm to $\infty$, or equivalently in this context, to $\aleph_1$.
\edefi
Recall the functional \ref{muk} introduced in Section \ref{lll}. A functional $F$ of type 2 is called \emph{normal} if $\exists^2$ is computable in $F$, a definition that works for all Kleene closed type structures. If $F$ is normal, we can use $\exists^2$ and the recursion theorem to prove the following theorem, originally due to  Gandy \cite{supergandy}.
\begin{theorem}[Stage comparison] 
If $F^{2}$ is normal, there is an $F$-computable function $P$ such that if $\{e\}(F, \vec f , \vec a)$ and $\{d\}(F , \vec g , \vec b)$ are two alleged computations, then $P(\langle e , \vec f , \vec a , d , \vec g , \vec b\rangle) $ terminates if and only if at least one of the two alleged computations terminate, and then $P$ decides if $||\langle F , \vec f , \vec a\rangle|| < ||\langle F , \vec g , \vec b \rangle||$ or not. 
\end{theorem}
A `soft' consequence of stage comparison is \emph{Gandy selection}, first proved in \cite{supergandy}.  We only state the version we need in this paper.
Intuitively, $\lambda f.\{d\}(F, f)$ is a (partial) choice function with the biggest possible domain.
\begin{theorem}[Gandy Selection]\label{gandyselection}
Let $F^{2}$ be normal. Let $A \subset \N \times \N^\N$ and $e$ be such that $(a,f) \in A$ if and only if $\{e\}(F,f,a)$ terminates \($A$ is semi-computable in $F$\). Then there is an index $d$ such that $\{d\}(F,f)$ terminates if and only if there exists $a\in\N$ such that $(a,f) \in A$, and then $\{d\}(F,f)$ is one of these numbers.
\end{theorem}
\begin{proof}
The proof makes use of the recursion theorem, and the idea is first to select the nonempty \emph{computable} set of those $a$ leading to a computation $\{e\}(F,f,a)$ of minimal ordinal rank, and then to select the least $a$ among those. We omit further details and refer to \cite{longmann}*{\S5.4} instead.
\end{proof}
Most of our applications of Gandy selection are based on the following corollary. 
Intuitively speaking, the functional $G$ computes an $F$-index for $f$.
\begin{corollary}\label{cor.inj} 
Let $F^{2}$ be normal. Then there is a partial functional $G$ computable in $F$ which terminates if and only if the input $f$ is computable in $F$, and such that we have $G(f) = e\di (\forall a)(f(a) = \{e\}(F,a))$.  
\end{corollary}
\begin{proof} When $F$ is normal, the relation $(\forall a \in \N)(f(a) = \{e\}(F,a))$ is clearly\\ semi-computable, and we can apply Gandy Selection.\end{proof}
Note that the functional $G$ is always injective.
Of course, these results are equally valid for all Kleene closed type structures, and we may replace $\N^\N$ in Theorem \ref{gandyselection} with any finite product of $\N$ and $\N^\N$.

\subsubsection{Two type structures}\label{models}
In this section, we define two Kleene closed type structures ${\bf P}$ and $\bf Q$ that are crucial for the below independence results involving $\Z_{2}^{\omega}$.
Moreover, the construction of ${\bf Q}$ unifies constructions from earlier work that were used to prove e.g.\ \cite{dagsam}*{Theorem~3.4}, \cite{dagsamV}*{Theorem 4.9}, and \cite{dagsamVII}*{Theorem 3.5}.

\smallskip

We note that ${\bf P}$ is constructed under the set-theoretical assumption that $\textsf{V = L}$.  
There is no harm in this, since what is of interest is the logic of the structure, which statements are true and which are false, and our results will not depend on the assumption that $\textsf{V = L}$; they are proved within $\ZF$.
For the unfamiliar reader, the axiom $\textsf{V = L}$, which expresses that every set is constructible, was used by G\"odel to show that the Continuum Hypothesis is consistent with $\ZFC$ (see \cite{devil} for details).

\smallskip

Recall the functionals $\SS^{2}_{k}$ from Section \ref{lll}. We use the assumption $\textsf{V = L}$ motivated by the following fact from set theory.
\begin{lemma}[$\textsf{V = L}$] Let $A \subseteq \N^\N$ be closed under computability relative to all $\SS^{2}_{k}$. Then all $\Pi^1_n$-formulas are absolute for $A$ for all $n$. \end{lemma}
\begin{proof} 
For $n \leq 2$, this is a general fact independent of the assumption \textsf{V = L}, and for $n > 2$ it  is a consequence of the existence of a $\Delta^1_2$-well-ordering of $\N^\N$. 
\end{proof}
\begin{definition}[$\textsf{V = L}$]\label{defP}{\em 
Let $\SS^{2}_{\omega}$ be the join of all the functionals $\SS^{2}_{k}$, and let ${\bf P}$ be the Kleene closed  type-structure, as obtained from Theorem \ref{thm.ext}, where ${\bf P}[1]$ is the set of functions computable in $\SS^{2}_{\omega}$ and the restriction of $\SS^{2}_{\omega}$ to ${\bf P}[1]$ is in ${\bf P}[2]$.}
\end{definition}
The model ${\bf P}$, under another name, has been used to prove \cite{dagsamV}*{Theorem 4.3}.
Recall the unique extension $\TP^{*}$ of $\TP$ introduced below Definition \ref{uneek}.
\begin{lemma}\label{lemmaP}
 ${\bf P}^*$ derived  from ${\bf P}$ as defined above is a model for $\Z_{2}^{\omega} + \QFAC^{0,1}$. 
 \end{lemma}
\begin{proof} 
We assume that $\textsf{V = L}$, which implies that all $\Pi^1_n$-formulas are absolute for ${\bf P}[1]$.  Since ${\bf P}[1]$ is closed under computability relative to each $\SS^{2}_k$, we have that ${\bf P}[1]$ satisfies all $\Pi^1_n$-comprehension axioms. Now assume that $(\forall n^{0}) (\exists f^{1}) Q(n,f,\vec \Phi)$  is true in ${\bf P}$, where $Q$ is quantifier-free and $\vec \Phi$ is a list of parameters from ${\bf P}$. 
 Since all functionals in $\vec \Phi$ are computable in $\SS^{2}_\omega$, the set 
 \[
 R=\{(n, e)\in \N^{2} : (\exists f^{1})[ Q(n,f,\vec \Phi)\wedge (\forall a^{0})(f(a) = \{e\}(\SS_2^\omega,a))] \}
 \]
 is semi-computable in $\SS^{2}_{\omega}$. Moreover, we have that  $(\forall n^{0}) (\exists e^{0})[(n,e)  \in R]$.
 By assumption and Gandy selection, there is a function $g$ computable in $\SS^{2}_{\omega}$ such that $R(n,g(n))$ for all $n^{0}$. 
 If $G(n)$ is the function $f$ computed from $\SS^{2}_{\omega}$ with index $g(n)$, we have that $G^{0\di 1} \in \bf P$ witnesses this instance of quantifier-free choice.
 \end{proof} 
For the rest of this section, we fix some notation. We let $A \subseteq \N^\N$ be a countable set such that all $\Pi^1_n$ formulas are absolute for $(A, \N^\N) $ for all $n$. We let $A = \{g_k : k \in \N\}$ and we let $A_k$ be the set of functions computable in $\SS^{2}_k$ and $g_0 , \ldots , g_{k-1}$. For the sake of unity, we put $\SS^2_0 = \exists^2$, so $A_0$ is the set of hyperarithmetical functions.
 \begin{lemma} Each $A_k$ is a subset of $A$. Moreover, 
 for each $k$, $A_k$ is a proper subset of $A_{k+1}$ and $A_{k+1}$ contains an element $h_k$ that enumerates $A_k$.\end{lemma}
 \begin{proof} This follows from the choice of $A$ and the fact that the relation \[\{e\}(\SS^{2}_{k} , g_0 , \ldots , g_{k-1}, \vec a) = b\] is $\Pi^1_{k+1}$ (for $k > 0$ even $\Delta^1_{k+1}$) and that $\SS^2_{k+1}$ computes an enumeration of all functions computable in $\SS^2_k$ relative to any fixed list of type 1 arguments.  
 \end{proof}
Now fix $F:A \rightarrow \N$. We intend to use Theorem \ref{thm.ext} and let $F \in B_{0}$ if there is a $k_0$ such that for all $k \geq k_0$, the restriction of $F$ to $A_k$ is partially computable in $\SS^{2}_k$ and $g_0 , \ldots , g_{k-1}$. 
Note that the join of finitely many functionals from $B_{0}$ is in $B_{0}$.
\begin{lemma} 
If $F \in B_{0}$ and $f$ is computable from $F$ and elements $\vec f$ in $A$, then $f \in A$. Moreover $\SS^2_k$ restricted to $A$ is in $B_{0}$.
\end{lemma}
\begin{proof} 
For the first item, choose $k_0$ for $F$ and $k \geq k_0$ so large that $\vec f$ is a sequence from $A_k$. 
By Lemma \ref{lemma.blowout} we have that $f \in A_k$. For $l \geq k$ we have that $\SS^2_k$ restricted to $A_l$ is computable in $\SS^2_l$. 
This shows the second item.
\end{proof}
\begin{definition}\label{defQ}
{\em We define $\bf Q$ to be the Kleene closed typed structure obtained by applying Theorem \ref{thm.ext} to $A$ and $B_{0}$ as given above.}
\end{definition}
Recall the unique extension $\TP^{*}$ of $\TP$ introduced below Definition \ref{uneek}.
The type structure ${\bf Q}^*$ is then a model for $\Z_2^\omega$ since all $\Pi^1_k$-statements are absolute for $A$ for all $k$. 
In the next section, we show that ${\bf Q}^*$ does not satisfy $\QFAC^{0,1}$, while we use $\bf Q$ to show that $\Z_2^\omega$ is consistent with $\neg\NBI$ (see Theorem \ref{proleet}).

\section{The uncountability of $\R$ in Reverse Mathematics}\label{main}
\subsection{Introduction}
In this section, we study $\NIN$ and $\NBI$ in higher-order RM, as sketched in Section \ref{prelim1}.
Our results are summarised by the following list. 
\begin{itemize}
\item We calibrate how much comprehension and/or choice proves $\NIN$ and $\NBI$ in higher-order RM (Sections \ref{hogr} and \ref{uncountsyn2}).
\item We curate a large collection of third-order principles $T$ such that  $T \di \NIN$ where $T$ is a weak theorem of ordinary mathematics, like Arzela's convergence theorem for the Riemann integral (Sections \ref{forgs} and \ref{metri}).
\item We explore different notions of countability, namely higher-order definitions of countability closely related to $\NIN$ (Definition \ref{standard}) and to $\NBI$ (Definition \ref{standard2}), as well the definition from second-order RM.  This includes a study of the `coding principles' $\cocode_{i}$ for $i=0,1$ which connects the second- and third order notions of countability (Section \ref{uncountsyn} and \ref{uncountsym2}).
\item We identify an `explosive' third-order principle; the latter looks harmless when formulated in a second-order setting, but is extremely strong when combined with higher-order comprehension functionals. In particular, the combination of the Bolzano-Weierstrass theorem for countable sets and the Suslin functional together results in a blow-up to $\Pi^1_2\textup{\textsf{-CA}}_0$ (Section \ref{radioexplo}).
\end{itemize}
Independence results shall be proved using the models from Section \ref{models}.    

\smallskip

Finally, we mention a little `trick' that is convenient when proving $\NIN$ or $\NBI$.
First observe that $\NBI$ and $\NIN$ are trivial if \emph{all functions on $\R$ are continuous everywhere}.  
Now, the latter statement in italics is equivalent to $\neg(\exists^{2})$, which follows from Theorem \ref{Xxn} by contraposition (and classical logic). 
\begin{thm}[\cite{kohlenbach2}*{Prop.\ 3.12}]\label{Xxn}
The following are equivalent over $\RCAo$:
\begin{itemize}
\item the axiom $(\exists^{2})$,
\item there exists a function $F:\R\di \R$ that is not continuous at some $x\in \R$.
\end{itemize}
\end{thm}
Now, in light of the law of excluded middle as in $(\exists^{2})\vee \neg(\exists^{2})$, we 
may always assume $(\exists^{2})$ when proving $\NIN$ or $\NBI$.  Indeed, in case $\neg(\exists^{2})$, $\NIN$ and $\NBI$ are trivial.  
In the sequel, we will often do so without any further comment.  

\subsection{The principle $\NIN$}
We establish the properties of $\NIN$ summarised in Figure~\ref{dd}.  
In terms of comprehension and normal functionals, we show that $\NIN$ is not provable in $\Z_{2}^{\omega}$ but provable in $\Z_{2}^{\Omega}$, 
where the latter two are both conservative extensions of $\Z_{2}$ (see Section \ref{hogr}).
On the other hand, $\NIN$ already follows from rather basic (non-normal) mathematical facts, including the following list, 
while another dozen of basic theorems implying $\NIN$ are listed in Appendix \ref{extraextrareadallaboutit}.
\begin{itemize}
\item covering theorems (Heine-Borel, Lindel\"of, and Vitali) about \emph{uncountable coverings} (Section \ref{forgs}),
\item a basic version of the Baire category theorem (Section \ref{forgs}).  
\item basic properties of the Riemann integral (Theorem~\ref{arzen}, Arzel\`a 1885 \cite{arse2})
\item basic properties of metric spaces (Section \ref{metri}), 
\item basic theorems from RM about countable sets (Sections \ref{uncountsyn} and \ref{radioexplo}).
\end{itemize}
By the first two items, the negative results concerning e.g.\ $\HBU$ and $\WHBU$ from \cites{dagsamIII, dagsamVI, dagsamVII} follow from the properties of $\NIN$ proved in this paper, i.e.\ we reprove many of our previous results \emph{in one fell swoop}.  Regarding the final item, the Bolzano-Weierstrass theorem for \emph{countable sets} in $2^{\N}$ is weak but gives rise to $\SIX$ when combined with higher-order $\FIVE$, i.e.\ the Suslin functional (see Theorem \ref{BOOM}).

\smallskip

We stress
the results in \cite{kohlenbach2}*{\S3} which establish the equivalence \textbf{over} $\RCAo$ between $(\exists^{2})$ and the existence of a discontinuous function on $\R$.  
Since $\NIN$ is trivial if all functions on $\R$ are continuous, we have $\neg\NIN\di(\exists^{2})$ by contraposition, a fact we will often make use of \emph{without further comment}.

\subsubsection{Comprehension and $\NIN$}\label{hogr}
In this section, we show that $\NIN$ relates to comprehension as follows.
Recall that $\Z_{2}^{\omega}$ and $\Z_{2}^{\Omega}$ are conservative extensions of $\Z_{2}$. 
\begin{thm}\label{dick}
$\Z_{2}^{\omega} + \QFAC^{0,1}$ cannot prove $\NIN$, while $\Z_{2}^{\Omega}+\QFAC^{0,1}$ can. 
\end{thm}
\begin{proof} 
For the negative result, we use the model ${\bf P}$ from Definition \ref{defP}. This model satisfies $\Z_2^\omega + \QFAC^{0,1}$. 
We observe that ${\bf P}[2]$ contains a functional $G:A \rightarrow \N$ that is \emph{injective}, which is a direct consequence of Corollary \ref{cor.inj} using $\SS^{2}_\omega$ for $F$.

\smallskip

For the positive result, fix $Y:[0,1]\di \N$ and consider the following formula, which is trivial under classical logic:
\be\label{triv}
(\forall n\in \N)(\exists y\in [0,1])\big[  (\exists x\in [0,1])(Y(x)=n)\di Y(y)=n \big].
\ee
Note that $\QFAC^{0,1}$ applies to \eqref{triv} (modulo $\exists^{3}$) and let $\Phi^{0\di 1}$ be the resulting sequence, which obviously lists the range of $Y$.  
Using \cite{simpson2}*{II.4.9}, let $y\in [0,1]$ be a real number such that $y\ne_{\R}\Phi(n)$ for all $n\in \N$.
For $n_{0}:=Y(y)$, we have $y\ne_{\R} \Phi(n_{0})$ while $Y(y)=n_{0}=Y(\Phi(n_{0}))$, i.e.\ $\NIN$ follows.
\end{proof}
An elementary argument allows us to replace $(\exists^{3})$ by $\BOOT$ and restrict to the Axiom of \emph{unique} Choice in the previous, but we obtain a much sharper proof in Theorem \ref{flahu}; the theorem holds if we restrict $\NIN$ to \emph{measurable} functionals.  The proof also establishes that the Axiom of Countable Choice suffices to prove $\NIN$, 
while the same holds for (many?) other fragments, like HR35 and HR38 from \cite{ConAC}.

\smallskip

We finish this section with a remark on the (very recent) RM of $\NIN$.
\begin{rem}[Equivalences for $\NIN$]\label{lafke}\rm
The higher-order RM of $\NIN$ has recently been studied in \cite{samcie22}, yielding two kinds of results, over a suitable base theory.
\begin{itemize}
\item The principle $\NIN$ is equivalent to the statement that there is no injection from $X$ to $\N$, for $X$ equal to $\R, 2^{\N},$ or $ \N^{\N}$.
\item The principle $\NIN$ is equivalent to $\NIN_{\Y}$, i.e.\ the statement that there is no injection $Y:[0,1]\di \Q$ with $Y$ in a certain function class $Y$.
\end{itemize}
While the first item is fairly basic/technical in nature, the second item is established in \cite{samcie22} for the function classes $\Y$ based on the notion of bounded variation, semi-continuity, the Sobolev space $W^{1, 1}$, cliquishness, and Borel, all formulated in the language of third-order arithmetic.
In this light, the hardness of $\NIN$ as in Theorem \ref{dick} is \textbf{not} due to the quantification over \emph{arbitrary} $\R\di \N$-functions in $\NIN$.   
\end{rem}

\subsubsection{Ordinary mathematics and $\NIN$}\label{forgs}
In this section, we derive $\NIN$ from weak\footnote{When added to $\RCAo$, e.g.\ $\WHBU$ and $\BCT$ from Theorem \ref{flahu} do not increase the first-order strength of the former.  The results in \cite{kruisje}*{Theorem 3} show that $\WHBU$ yields a conservative extension of arithmetical comprehension, as provided by Feferman's $\mu$ from Section \ref{lll}.} theorems of ordinary mathematics, including Arzel\`a's convergence theorem $\Arz$.  
We also note that countable choice as in $\QFAC^{0,1}$ is not needed to prove $\NIN$.

\smallskip

Let $\BCT$ be the Baire category theorem for open sets given by characteristic functions as in \cite{dagsamVII}*{\S6}.
The following proof still goes through if we further restrict $\BCT$ to open sets with at most finitely many isolated points in the complement.  
\begin{thm}\label{flahu}
The system $\RCAo$ proves  $\WHBU\di \NIN\leftarrow \BCT$.  
\end{thm}
\begin{proof}
For the first result, let $Y:[0,1]\di \N$ be an injection, i.e.\ we have that $(\forall x, y\in [0,1])(Y(x)=_{0}Y(y)\di x=_{\R} y)$.  
Now consider the uncountable covering $\cup_{x\in [0,1]}B(x, \frac{1}{2^{Y(x)+3}})$ of $[0,1]$.  
Since $Y$ is an injection, we have $\sum_{i\leq k}|B(x_{i}, \frac{1}{2^{Y(x_{i})+3}})|\leq \sum_{i\leq k} \frac{1}{2^{i+2}} \leq \frac{1}{2}$ for any finite sequence $x_{0}, \dots, x_{k}$ of distinct reals in $[0,1]$.    
In this light, $\WHBU$ is false and we obtain $\WHBU\di \textsf{NIN}$, as required. 

\smallskip

For the second part, let $Y:[0,1]\di \N$ again be an injection.  
Now define $O_n = \{x\in [0,1] : Y(x) > n\}$ and note that since the complement of each $O_n$ is finite, each $O_n$ is open and dense. Moreover, $\{(n,x) : x \in O_n\}$ is definable from $Y$ by a term in G\"odel's $T$, so this is a countable sequence of dense, open sets in $[0,1]$. The intersection is empty and $\BCT$ thus fails; $\BCT\di \NIN$ now follows. 
\end{proof}
We note that the Heine-Borel and Vitali theorems for \emph{countable} coverings similarly imply $\NIN$, as shown in Section \ref{uncountsyn}.
Since $\BCT$ for open sets given as countable unions (aka RM-codes) is provable in $\RCA_{0}$ (\cite{simpson2}*{II.4.10}), $\NIN$ also follows from 
the `coding principle' that expresses that a sequence of open sets given by characteristic functions as in \cite{dagsamVII}*{\S6} can be expressed as a sequence of RM-codes of open sets.  
Moreover, the second part of this proof can be combined with Lemma~\ref{lemma.blowout} to yield a simpler proof of \cite[Theorem~6.6]{dagsamVII} as done in Theorem \ref{triplez}. 

\smallskip

As noted above and in \cite{dagsamVI}, $\WHBU$ constitutes the combinatorial essence of the Vitali covering theorem.  The former is equivalent to fundamental properties of the Lebesgue measure, like countable additivity, 
over a slight extension of $\RCAo$ to accommodate basic measure theory (\cite{dagsamVI}).   In Section \ref{extraextrareadallaboutit},  we list some \emph{very} basic properties of the Lebesgue measure and integral that imply $\NIN$.  
A \emph{much} more basic `integral' theorem that implies $\NIN$ is provided by Arzel\`a's convergence theorem \emph{for the Riemann integral}, first published in 1885 (\cite{arse2}) and discussed in (historical) detail in \cites{luxeternam,kakeye}.  The proof in \cite{thomon3} readily yields $\Arz$ in $\Z_{2}^{\Omega}$, using $\HBU$, while Yokoyama studies $\Arz$ for continuous functions in second-order RM in \cite{yokoyamaphd}*{Theorem 3.33}.   
\begin{princ}[$\Arz$]
Let $f$ and $(f_{n})_{n\in \N}$ be Riemann integrable on the unit interval and such that $\lim_{n\di \infty}f_{n}(x)=f(x)$ for all $x\in [0,1]$.  
If there is $M\in \N$ such that $|f_{n}(x)|\leq M+1$ for all $n\in \N$ and $x\in [0,1]$, then $\lim_{n\di \infty}\int_{0}^{1}f_{n}(x)dx=\int_{0}^{1}f(x)dx$.
\end{princ}
As is clear from its proof, the following theorem does not change if we require a modulus of convergence for $\lim_{n\di \infty}f_{n}=f$, (universal) moduli of Riemann integrability, or if we assume the sequence $\lambda n.\int_{0}^{1}f_{n}(x)dx$ to be given.  
\begin{thm}\label{arzen}
The system $\RCAo$ proves $\Arz\di \NIN$.
\end{thm}
\begin{proof}
Let $Y:[0,1]\di \N$ be an injection, define $f$ as the constant $1$ function, and define $f_{n}(x):=\sum_{i\leq n}g_{i}(x)$, where $g_{i}(x)=1$ if $Y(x)=i$ and $0$ otherwise.  
Note that $f_{n}(x)$ can only be $0$ or $1$ due the injectivity of $Y$.  Moreover $f_{n}(x)=1$ for $n\geq Y(x)$, i.e.\ $\lim_{n\di \infty}f_{n}(x)=f(x)$ for all $x\in [0,1]$ (with a modulus of convergence).  
Clearly, $f$ is Riemann integrable on the unit interval, while the same holds for $f_{n}$ for fixed $n$.  Indeed, $g_{i}$ is either identical $0$ or zero everywhere except at (the unique by assumption) $x_{0}\in [0,1]$ such that $Y(x_{0})=i$, where $g_{i}(x_{0})=1$.  Hence, $f_{n}(x)$ has at most $n+1$ points of discontinuity.  Clearly, we have $\int_{0}^{1}f_{n}(x)dx=0$ and $\int_{0}^{1}f(x)dx =1$.
All conditions of $\Arz$ are satisfied, yielding $0=\lim_{n\di \infty}\int_{0}^{1}f_{n}(x)dx=\int_{0}^{1}f(x)dx=1$, a contradiction.  
\end{proof}
\noindent
The following theorems yield $\NIN$ in the same way as for $\Arz$ in the previous proof.
\begin{itemize}
\item Arzel\`a's lemma, called `Theorem B' in \cite{luxeternam} and `item 2' in \cite{gorkoen}*{Theorem 4}.
\item Luxemburg's \emph{Fatou's lemma for the Riemann integral} as in \cite{luxeternam}*{p.\ 977}.
\item Thomson's \emph{monotone convergence theorem for the Riemann integral} (\cite{thomon3}).
\item The Carslaw and Young \emph{term-by-term Riemann integration theorems} (see \cite{koolsla}*{Theorem II} and \cite{young1}).
\item Kestelman's \emph{Cauchy-Riemann convergence theorem} (\cite{kesteisdenbeste}*{Theorem 2}).
\item The above formulated using `continuous almost everywhere and bounded' by Lebesgue's criterion for Riemann integrability (see also Section \ref{uncountsyn}). 
\item Helly's convergence theorem for the Stieltjes integral (\cite{hellyeah}*{VIII, p.\ 288}).
\end{itemize}
As discussed in \cite{kakeye}*{\S4.4} in the context of Fourier series, Ascoli, Dini, and du Bois Reymond already made use of term-by-term integration for the Riemann integral involving discontinuous functions as early as 1874.  
Moreover, if one requires in \emph{Helly's selection theorem} that the sub-sequence exhibits pointwise convergence (like in the original \cite{hellyeah}*{VII, p.\ 283}) \emph{and} $L_{1}$-convergence (like in \textsf{HST} in \cite{kreupel}), then this version, which can be found in e.g.\ \cite{barbu}, implies $\NIN$ in the same way as in the theorem.  By contrast, Helly's selection theorem involving codes (called $\textsf{HST}$ in \cite{kreupel}) is equivalent to $\ACA_{0}$.

\smallskip

The previous theorem has a rather remarkable corollary.  Indeed, Tao's notion of metastability generally has nicer\footnote{As shown in \cite{kohlenbach3}*{p.\ 31}, a monotone sequence in the unit interval has an elementary computable and uniform rate of metastability, provable in $\RCAo$.  By contrast, it follows from \cite{simpson2}*{III.2} that $\ACA_{0}$ is equivalent to the convergence of monotone sequences in $[0,1]$.} 
logical/computational properties than `usual' convergence to a limit (\cite{taote, kohlenbach3}).   
The situation is rather different for $\Arz$.
\begin{cor}
The theorem remains valid if we change the conclusion of $\Arz$ to:
\begin{itemize}
\item The limit $\lim_{n\di \infty}\int_{0}^{1}f_{n}(x)dx$ exists.
\item The sequence $\lambda n.\int_{0}^{1}f_{n}(x)dx$ is metastable\footnote{A sequence of real numbers $(x_{n})_{n\in \N}$ is called \emph{metastable} if it satisfies $(\forall \eps>0, g:\N\di \N)(\exists N\in \N)(\forall n, m\in [N, g(N)])(|x_{n}-x_{m}|<\eps)$.}. $(\Arz^{-})$
\end{itemize}
\end{cor}
\begin{proof}
Let $Y:[0,1]\di \N$ be an injection and let $g_{i}(x)$ and $f(x)$ be as in the proof of the theorem.  Now define $f_{n}(x)$ as follows:
\[
f_{2n+1}(x):=
\begin{cases}
1 & 0\leq x\leq \frac{1}{4}\\
\sum_{i=0}^{2n+1} g_{i}(x) & \textup{otherwise}
\end{cases}
\textup{ and }
f_{2n}(x):=
\begin{cases}
1 & \frac{7}{8}\leq x\leq 1\\
\sum_{i=0}^{2n} g_{i}(x) & \textup{otherwise}
\end{cases}.
\]
In the same way as in the proof of Theorem \ref{arzen}, we have $\lim_{n\di \infty}f_{n}=f$ (with a modulus of convergence), $\int_{0}^{1}f_{2n+1}(x)dx =\frac{1}{4}$, and $\int_{0}^{1}f_{2n}(x)dx =\frac{1}{8}$.
Hence, any one of the conditions of the corollary leads to a contradiction, and $\NIN$ follows. 
\end{proof}
The previous proof goes through for $g$ in the definition of metastability restricted to constant functions (see e.g.\ \cite{kohlenbach3}*{p.\ 499} for such results `in the wild'). 
Moreover, the above results have clear implications for the `coding practise' of RM, which we shall however discuss elsewhere in detail. 

\smallskip

Finally, on a conceptual note, a number of early critics (including Borel) of the \emph{Axiom of Choice} actually implicitly used this axiom in their work (see \cite{nofega}*{p.~315}).  A similar observation can be made for $\NBI$ and $\NIN$ as follows:  around 1874, Weierstrass seems to have held the belief\footnote{Weierstrass seems to have changed his mind by 1885, which he expressed in a letter to Mittag-Leffler (see \cite{nofega}*{p.\ 185}).} that there cannot be essential differences between infinite sets (see \cite{nofega}*{p.\ 184}), 
although basic compactness results, pioneered in part by Weierstrass himself, imply the uncountability of $\R$.  

\subsubsection{Metric spaces and $\NIN$}\label{metri}
In this section, we derive $\NIN$ from a most basic separability property of metrics on the unit interval. 

\smallskip

Now, the study of metric spaces in RM proceeds -unsurprisingly- via codes, namely a complete separable metric space is represented via a countable and dense subset (\cite{simpson2}*{II.5.1}). 
It is then a natural question how hard it is to prove that this countable and dense subset exists for the original/non-coded metric spaces. 
We study the special case for metrics \emph{defined on the unit interval}, as in Definition \ref{donc} and $\STS$ below, which implies $\NIN$ by Theorem \ref{STS}.  
\bdefi\label{donc}
A functional $d: [0,1]^{2}\di \R$ is a \emph{metric on the unit interval} if it satisfies the following properties for $x, y, z\in [0,1]$:
\begin{enumerate}
 \renewcommand{\theenumi}{\alph{enumi}}
\item $d(x, y)=_{\R}0 \asa  x=_{\R}y$,
\item $0\leq_{\R} d(x, y)=_{\R}d(y, x), $
\item $d(x, y)\leq_{\R} d(x, z)+ d(z, y)$.
\end{enumerate}
We use standard notation like $B_{d}(x, r)$ to denote $\{y\in [0,1]: d(x, y)<r\}$.
\edefi
\bdefi[Countably-compact]
The metric space $([0,1], d)$ is \emph{countably-compact} if for any sequence $(a_{n})_{n\in \N}$ in $[0,1]$ and sequence of rationals $(r_{n})_{n\in \N}$ such that $[0,1]\subset \cup_{n\in \N}B_{d}(a_{n}, r_{n})$, there is $m\in \N$ such that  $[0,1]\subset \cup_{n\leq m}B_{d}(a_{n}, r_{n})$.
\edefi
We note that Definition \ref{SEPKE} is used in constructive mathematics (see \cite{troeleke2}*{Ch.\ 7, Def.\ 2.2}).  
Our notion of separability is also implied by \emph{total boundedness} as used in RM (see \cite{simpson2}*{III.2.3} or \cite{browner}*{p.\ 53}).
According to Simpson (\cite{simpson2}*{p.\ 14}), one cannot speak at all about non-separable spaces in $\L_{2}$.  
\bdefi[Separability]\label{SEPKE}
A metric space $([0,1], d)$ is \emph{separable} if there is a sequence $(x_{n})_{n\in \N}$ in $[0,1]$ such that $(\forall x\in [0,1], k\in \N)(\exists n\in \N)( d(x, x_{n})<\frac{1}{2^{k}})$.
\edefi
\begin{princ}[$\STS$]
A countably-compact metric space $([0,1], d)$ is separable.
\end{princ}
\begin{thm}\label{STS}
The system $\RCAo$ proves $\STS\di \NIN$.
\end{thm}
\begin{proof}
Recall that by \cite{kohlenbach2}*{\S3}, $\NIN$ trivially holds if $\neg(\exists^{2})$ as in the latter case all functions on $\R$ are continuous.  
Thus, we may assume $(\exists^{2})$ for the rest of the proof. 

\smallskip

Suppose $Y:[0,1]\di \N$ is an injection and define $d(x, y):=|\frac{1}{2^{Y(x)}}-\frac{1}{2^{Y(y)}}|$ in case $x, y\in [0,1]$ are non-zero.  Define $d(0, 0):=0$ and $d(x, 0)= d(0, x):= \frac{1}{2^{Y(x)}}$ for non-zero $x\in [0,1]$.
The first item in Definition~\ref{donc} holds by the assumption on $Y$, while the other two items hold by definition. 

\smallskip

The metric space $([0,1], d)$ is countably-compact as $0\in B_{d}(x, r)$ implies $y\in B_{d}(x, r)$ for $y\in [0,1]$ with only finitely many exceptions (due to $Y$ being an injection).  
Let $(x_{n})_{n\in\N}$ be the sequence provided by $\STS$, implying $(\forall x\in [0,1](\exists n\in \N)( d(x, x_{n})<\frac{1}{2^{Y(x)+1}})$ (by taking $k=Y(x)+1$).  
The latter formula implies $(\forall x\in [0,1](\exists n\in \N)( |\frac{1}{2^{Y(x)}}-\frac{1}{2^{Y(x_{n})}}|<\frac{1}{2^{Y(x)+1}})$ by definition.  
Clearly, $|\frac{1}{2^{Y(x)}}-\frac{1}{2^{Y(x_{n})}}|<\frac{1}{2^{Y(x)+1}}$ is only possible if $Y(x)=Y(x_{n})$, implying $x=_{\R}x_{n}$. 
Hence, we have shown that $(x_{n})_{n\in \N}$ lists all reals in the unit interval.
By \cite{simpson2}*{II.4.9}, there is $y\in [0,1]$ such that $y\ne x_{n}$ for all $n\in \N$ (in $\RCA_{0}$).  
This contradiction implies $\NIN$. 
\end{proof}
\begin{cor}
The theorem still goes through upon replacing `separable' in $\STS$ by any of the following.
\begin{enumerate}
 \renewcommand{\theenumi}{\alph{enumi}}
\item Total boundedness as in \cite{simpson2}*{III.2.3} or \cite{browner}*{p.\ 53}.   
\item The Heine-Borel property for \emph{uncountable} covers.\label{lappel}
\item The Lindel\"of property for \emph{uncountable} covers.\label{lappel2}
\item The Vitali covering property as in $\WHBU$ for \emph{uncountable} Vitali covers.  \label{frap}
\end{enumerate}
\end{cor}
\begin{proof}
For item \eqref{lappel}, fix an injection $Y:[0,1]\di \N$ and let $d$ be the metric as in the proof of the theorem.  
Then $B_{d}(x, r)=\{x\}$ for $r>0$ small enough and $x\in (0,1]$.  
In particular, the uncountable covering $\cup_{x\in [0,1]}B_{d}(x, \frac{1}{2^{Y(x)+1}})$ of $[0,1]$ cannot have a finite (or countable) sub-cover, and $\NIN$ follows. 

\smallskip

For item \eqref{frap} (and item \eqref{lappel2}), let $Y$ and $d$ be as in the previous paragraph and note that $\cup_{x\in [0,1]}B_{d}(x, \frac{1}{2^{Y(x)+1}})$ is a Vitali cover.  
By the above, no finite sum can be larger than $1/2$, and we are done. 
\end{proof}
It should also be straightforward to derive $\NIN$ from the \emph{non}-separability of e.g.\ the sequence space $\ell^{\infty}$ or the space $\textsf{BV}$ of functions of bounded variation.  

\subsubsection{Countable sets versus sets that are countable}\label{uncountsyn}
We derive $\NIN$ from basic theorems about \emph{countable sets} where the latter has its \emph{usual} meaning, namely Definition~\ref{standard} taken from \cite{kunen}.
Among others, we study the \emph{Lebesgue criterion for Riemann integrability} for countable sets, as well as central theorems from RM concering countable sets as in items \eqref{bokes}-\eqref{bokes3}.  

\smallskip

First of all, we use the usual definition of \emph{countable set}, as follows.
By item \eqref{koer} in Definition~\ref{keepintireal}, sets $A\subset \R$ are given by characteristic functions, as in \cites{kruisje, dagsamVI, dagsamVII, hunterphd}.     
\bdefi[Countable subset of $\R$]\label{standard}~
A set $A\subseteq \R$ is \emph{countable} if there exists $Y:\R\di \N$ such that $(\forall x, y\in A)(Y(x)=_{0}Y(y)\di x=_{\R}y)$. 
\edefi
This definition is from Kunen's textbook on set theory (\cite{kunen}*{p.\ 63}); we \emph{could} additionally require that $Y:\R\di \N$ in Definition \ref{standard} is also \emph{surjective}, as in e.g.\ \cite{hrbacekjech}.  
This stronger notion is called `strongly countable' (see Definition~\ref{standard2}) and studied in Section~\ref{bereft}.  
If we replace `countable' by `strongly countable', all the below proofs go through \emph{mutatis mutandis} for $\NIN$ replaced by $\NBI$.  

\smallskip

Now, a cursory search reveals that the word `countable' appears hundreds of times in the `bible' of RM \cite{simpson2}.  
Sections titles of \cite{simpson2} also reveal that the objects of study are `countable' rings, vector spaces, groups, et cetera.  
Of course, the above definition of `countable subset of $\R$' cannot be expressed in $\L_{2}$.  Indeed, all the 
aforementioned objects are given by \emph{sequences} in $\L_{2}$ (see also \cite{simpson2}*{V.4.2}).  
Thus, the following\footnote{It is a tedious-but-straightforward verification that the below proofs still go through if we replace the equivalence by a forward arrow in $\cocode_{0}$.  One `immediate' example is that $\cocode_{0}\di \HAR_{0}$ in the proof of Theorem \ref{kloothommel}.} 
 `coding principle' $\cocode_{0}$ is \emph{crucial} to RM \emph{if} one wants the results in \cite{simpson2} to have the same scope 
as third-order theorems about countable objects as in Definition \ref{standard}.  This is particularly true for the RM of topology from \cite{mummy,mummyphd, mummymf}, as this enterprise is based on \emph{countable} bases at its very core.
\begin{princ}[$\cocode_{0}$]
For any non-empty countable set $A\subseteq [0,1]$, there is a sequence $(x_{n})_{n\in \N}$ in $A$ such that $(\forall x\in \R)(x\in A\asa (\exists n\in \N)(x_{n}=_{\R}x))$.
\end{princ}
Coding principles for continuous functions\footnote{As an example, $\RCAo+\WKL$ can prove the coding principle `$\textsf{coco}$' that any third-order function on $2^{\N}$ satisfying the usual definition of continuity, has an RM-code (see \cite{kohlenbach4}*{\S4}).} are used or studied in e.g.\ \cites{samrep, kohlenbach4, dagsamVII}.
As it happens, a version of $\cocode_{0}$ for representations has been studied in the context of Weihrauch reducibility in the form of \textsf{List} and \textsf{wList} from \cite{search}*{\S6}).
By Theorem \ref{kloothommel}, a lot of comprehension is needed to prove $\cocode_{0}$, but then the latter is clearly\footnote{Note that $\cocode_{0}$ is trivial given $\neg(\exists^{2})$, just like e.g.\ the Lindel\"of lemma (\cite{dagsamIII, dagsamV}).  Indeed,  $\neg(\exists^{2})$ implies that all functions on $\R$ are continuous by \cite{kohlenbach2}*{\S3}.} non-normal.
We also discuss the following basic theorems, which have fairly trivial proofs when formulated in $\L_{2}$.  Around 1885, Harnack proves the following in \cite{harny}*{p.\ 243} (see \cites{stillerebron, stilstebron, mazout} for a critical discussion).
\begin{princ}[$\textsf{Harnack}_{0}$]
A countable set $A\subset [0,1]$ has measure zero.
\end{princ}
Tao formulates a \emph{pigeon hole principle for measure spaces} in \cite{taoeps}*{p.\ 91} and Principle \ref{phm} is a special case for $[0,1]$.  
To ensure that the union in $\PHM$ exists, we always assume $(\exists^{2})$.  
Note that $\PHM$ is a special case of $\textsf{CUZ}$ from \cite{dagsamVI}.
\begin{princ}[$\PHM$]\label{phm}
For a sequence of sets $E_{n}\subset [0,1]$, if $A=\cup_{n\in \N}E_{n}$ has positive measure, then there is $n_{0}\in \N$ such that $E_{n_{0}}$ has positive measure. 
\end{princ}
Note that these principles can be formulated without\footnote{For $A\subset \R$, let `$A$ has measure zero' mean that for any $\eps>0$, there is a sequence of closed intervals $\big(I_{n}\big)_{n\in \N}$ covering $A$ and such that $\eps >\sum^{\infty}_{n=0 }|J_{n}|$ for $J_{0}:= I_{0}$ and $J_{i+1}:= I_{i+1}\setminus \cup_{j\leq i}I_{j}$.  This is nothing more than the usual definition as used by Tao in e.g.\ \cite{taomes}*{p.\ 19}.\label{clukker}} mentioning the Lebesgue measure, like was done for $\WHBU$, and e.g.\ by Harnack himself in \cite{harny}.  The following is a special case of the \emph{Lebesgue criterion for the Riemann integral}; the latter was discovered before 1870 (\cite{hanky}*{p.\ 92}) with a correct proof in \cites{harny2, duboiske, dinipi}.  
Smith studies $\LEB_{0}$ for a \emph{sequence} of exceptional points in \cite{snutg}, motivated by \cite{hanky}.  
\begin{princ}[$\LEB_{0}$]
A bounded function $f:[0,1]\di \R$ which is continuous outside a countable set $A\subset [0,1]$, is Riemann integrable. 
\end{princ}
Now, the weakest comprehension principle that can prove $\cocode_{0}$ seems to be $\BOOT^{-}$ from \cite{dagsamVII}, which is a weakening of $\BOOT$ via the following extra condition:
\[
(\forall n\in \N)(\exists \textup{ at most one } f\in \N^{\N})(Y(f, n)=0).
\]
As discussed below, $\BOOT^{-}$ follows from basic theorems on open sets as in \cite{dagsamVII}. 
\begin{thm}\label{kloothommel}
The system $\RCAo+\WKL$ proves 
\be\label{cahoot}
\BOOT^{-}\di \cocode_{0}\di  \HAR_{0} \di\LEB_{0}\di \NIN, 
\ee
with $\WKL$ only used in the third implication; $\ACAo$ proves $\PHM\di \HAR_{0}$.  
\end{thm}
\begin{proof}
We note that all principles in \eqref{cahoot} (as well as $\PHM$ and $\HAR_{0}$) are outright provable in $\RCAo+\WKL$ if all functions on $\N^{\N}$ and $\R$ are continuous.  
Since the latter is the case given $\neg(\exists^{2})$ by \cite{kohlenbach2}*{\S3}, we may assume $(\exists^{2})$ for the rest of the proof.  
The functional $\exists^{2}$ allows us to (uniformly) convert between real numbers and their binary or decimal representations, which we will tacitly do. 

\smallskip

First of all, we show $\BOOT^{-}\di \NIN$; repeating the proof with $[0,1]$ replaced by countable $A\subset [0,1]$, one obtains $\BOOT^{-}\di \cocode_{0}$.  
Let $Y:[0,1]\di \N$ be an injection and use $\BOOT^{-}$ to define $X\subset \N^{2}\times \Q$ such that for all $n, m\in \N$ and $q\in \Q\cap [0,1]$, we have:
\be\label{kilop}\textstyle
(n, m, q)\in X\asa (\exists x\in B(q, \frac{1}{2^{m}})\cap [0,1])(Y(x)= n ).
\ee
Since $Y$ is an injection, the following condition, required for $\BOOT^{-}$, is satisfied: 
\[\textstyle
(\forall n, m\in \N, q\in \Q\cap [0,1])(\exists \textup{ at most one } x\in B(q, \frac{1}{2^{m}})\cap [0,1])(Y(x)= n ).
\]
We now use $X$ from \eqref{kilop} and the well-known interval-halving technique to create a sequence $(x_{n})_{n\in \N}$.  
For fixed $n\in \N$, define $[x_{n}](0)$ as $0$ if $(\exists x\in [0, 1/2))(Y(x)=n)$, and $1/2$ otherwise; define $[x_{n}](m+1)$ as $[x_{n}](m)$ if $(\exists x\in \big[[x_{n}](m), [x_{n}](m)+\frac{1}{2^{m+1}}\big)\big)(Y(x)=n)$, and $[x_{n}](m)+\frac{1}{2^{m+1}}$ otherwise.  By definition, we have:  
\[
(\forall n\in \N)\big[(\exists x\in [0,1])(Y(x)=n)\asa  Y(x_{n})=n \big].
\]
Now use \cite{simpson2}*{II.4.9} to find $y_{0}\in [0,1]$ not in the sequence $(x_{n})_{n\in \N}$.
Then $n_{0}=Y(y_{0})$ yields a contradiction as $x_{n_{0}}\ne_{\R} y $ and $Y(y_{0})=n_{0}=Y(x_{n_{0}})$.

\smallskip

Secondly, for the implication $\cocode_{0} \di \HAR_{0}$, given a countable set $A$ and a sequence $(x_n)_{n\in \N}$ listing its elements, consider $I_{n}:= (x_{n}-\frac{\eps}{2^{n+2}}, x_{n}+\frac{\eps}{2^{n+2}})$, which 
are as required to show that $A$ has measure $0$.  

\smallskip

Thirdly, for the implication $ \HAR_{0}\di \LEB_{0}$, one uses (the proof of) \cite{kohlenbach4}*{Prop.\ 4.7}
to show that a function $f:[0,1]\di \R$ as in $\LEB_{0}$ has a continuous modulus of continuity outside of $A$.
This yields an open covering of $[0,1]\setminus A$, while $\HAR_{0}$ provides an open covering of $A$.
Both coverings are given by sequences.  The proof of \cite{sayo}*{Theorem 10} is now readily adapted to yield that $f$ is Riemann integrable, as required by $\LEB_{0}$.

\smallskip

Fourth, let $Y:[0,1]\di \N$ be an injection and note that $A\equiv[0,\frac{1}{2}]$ is countable by Definition \ref{standard}.
Define $f:[0,1]\di \R$ as $2$ if $x\in (\frac12, 1]$ and the indicator function of $\Q$ otherwise.  By $\LEB_{0}$, this function is Riemann integrable, a contradiction.      

\smallskip

Fifth, assume $\PHM$ and suppose $E\subset [0,1]$ is countable and not measure zero.  
For $Y:[0,1]\di \N$ an injection on $E$, define $E_{n}:=\{x\in E:Y(x)=n\}$ and let $n_{0}$ be such that $E_{n_{0}}$ has positive measure.  
By the definition of measure zero, there must be at least two distinct $x, y\in E_{n_{0}}$, a contradiction as $Y(x)=n_{0}=Y(y)$.
\end{proof}
We now formulate a nice corollary involving $\cocode_{0}$ and its ilk restricted to closed sets as used in RM.
Now, open sets are given in RM by unions $\cup_{n\in \N}(a_{n}, b_{n})$ (\cite{simpson2}*{II.5.6}), while closed sets are complements thereof.  
We refer to such sets as `RM-open' and `RM-closed'.  By the following corollary, $\Z_{2}^{\omega}+\QFAC^{0,1}$ cannot prove that countable RM-closed sets are given by a sequence.  
\begin{cor}\label{frex}
The restriction of $\cocode_{0}$, $\HAR_{0}$, $\PHM$, or $\LEB_{0}$ to RM-closed sets $A$ still implies $\NIN$.
\end{cor}
\begin{proof}
Assume $\cocode_{0}$ for RM-closed sets and let $Y:[0,1]\di \N$ be an injection.  
The set $[0,1]$ is clearly RM-closed, as well as countable in the sense of Definition~\ref{standard}.
Hence, there is a sequence $(x_{n})_{n\in \N}$ listing all elements of $[0,1]$.  By \cite{simpson2}*{II.4.9}, there is $y\in [0,1]$ such that $(\forall n\in \N)(x_{n}\ne_{\R}y)$, a contradiction.
The other results follow in the same way.
\end{proof}
\noindent
One can show that $\Arz$ implies the restriction of $\PHM$ to RM-closed sets, while the latter restriction makes elementhood for $A=\cup_{n\in \N}E_{n}$ in $\PHM$ decidable.  

\smallskip

One can push the previous corollary even further as follows: $\NIN$ follows from the statement \emph{an RM-closed and countable subset of $\R$ has measure $<+\infty$.}
Other theorems that imply $\NIN$ in the same way are as follows, where `countable' is always interpreted as in Definition \ref{standard}.  
\begin{enumerate}
\renewcommand{\theenumi}{\roman{enumi}}
\item Heine-Borel theorem for countable collections of open intervals.\label{bokes} 
\item Vitali's covering theorem for countable collections of open intervals.\label{bokes2} 
\item Riemann integrable functions differing on countable sets have equal integral.
\item $|b-a|$ is the measure of: $[a,b]$ plus a countable set (cf.\ Footnote~\ref{clukker}).  
\item Convergence theorems for \textbf{nets} in $[0,1]$ with countable index sets. 
\item For a countable set, the Lebesgue integral of the indicator function is zero.
\item Ascoli-Arzel\`a theorem for countable sets of functions (see e.g. \cite{fourchette}).
\item Bolzano-Weierstrass: a countable set in $[0,1]$ has a supremum.\label{bokes3}
\item Topology formulated with countable bases as in Example \ref{mummyrem}.
\end{enumerate}
Regarding item \eqref{bokes}, Borel in \cite{opborrelen2} uses `countable infinity of intervals' and \textbf{not} `sequence of intervals' in his formulation\footnote{In fact, Borel's proof of the Heine-Borel theorem in \cite{opborrelen2}*{p.\ 42} starts with: \emph{Let us enumerate our intervals, one after the other, according to whatever law, but determined}.  He then proceeds with the usual `interval-halving' proof, similar to Cousin in \cite{cousin1}.} of the Heine-Borel theorem.  
Vitali similarly talks about countable and uncountable `groups' of intervals in \cite{vitaliorg}.  

\smallskip

As a corollary to Theorem \ref{flahu}, we now derive $\NIN$ from item \eqref{bokes} as follows.  
The Heine-Borel theorem for different representations of open coverings is studied in RM (\cite{paultoo}), i.e.\ the motivation for $\HBC_{0}$ is already present in second-order RM.
\begin{princ}[$\HBC_{0}$]
For countable $A\subset \R^{2}$ with $(\forall x\in [0,1])(\exists (a, b)\in A)(x\in (a, b))$, there are $(a_{0}, b_{0}), \dots (a_{k}, b_{k})\in A$ with $(\forall x\in [0,1])(\exists i\leq k)(x\in (a_{i},b_{i} ))$.
\end{princ}
\begin{cor}\label{forgot}
The system $\RCAo$ proves $\HBC_{0}\di \NIN$.
\end{cor}
\begin{proof}
Let $Y:[0, 1]\di \N$ be an injection and note that $[0,1]$ is now countable as in Definition \ref{standard}.
Fix $\Psi:[0,1]\di \R^{+}$ and define\footnote{It is a non-trivial exercise to show that the set $A$ can in fact be defined using $\exists^{2}$.} $A$ as the countable set $\{ (x-\Psi(x), x+\Psi(x)):x\in [0,1] \}$.   
Applying $\HBC_{0}$, there is a finite sub-cover and $\HBU$ follows.  Since $\HBU\di \WHBU$, Theorem \ref{flahu} yields $\NIN$, a contradiction. 
\end{proof}
If $\BW_{0}$ denotes item \eqref{bokes3} above, the `usual' proof yields $\BW_{0}\di \HBC_{0}$, as expected. 
The following corollary is immediate and shows the limitations of $\L_{2}$.
\begin{cor}
The system $\Z_{2}^{\omega}+\QFAC^{0,1}$ cannot prove $\HBC_{0}$ or $\BW_{0}$.
\end{cor}
\noindent
Finally, we study the RM of $\BW_{0}$, $\HBC_{0}$, etc.\ in \cite{dagsamXI}.  
We do establish the aforementioned `explosion' involving $\SIX$ and the Bolzano-Weierstrass theorem (for countable sets) in the next section.

\subsubsection{An explosive result}\label{radioexplo}
In this section, we establish Theorem \ref{BOOM} which expresses that the Bolzano-Weierstrass theorem for countable sets in $2^{\N}$, namely $\BW_{0}^{C}$ as in Definition \ref{comp}, is highly explosive: it yields $\SIX$ when combined with higher-order $\FIVE$, i.e.\ the Suslin functional $\SS^{2}$.  

\smallskip

First of all, we briefly motivate the importance of Theorem \ref{BOOM} as follows: the results in \cites{dagsamXI,dagsamXII, dagsamXIII} establish equivalences between $\BW_{0}^{C},$ $\cocode_{0}$, and the \emph{Jordan decomposition theorem} over a suitable base theory.  The latter theorem does not mention countable sets and has been classified at the level of $\ACA_{0}$ in second-order RM (\cites{nieyo, kreupel}), i.e.\ the third-order version behaves \emph{quite} differently. 
We discuss the further implications of Theorem~\ref{BOOM} in Remark \ref{pimp} and the details of the formalisation of $\BW_{0}^{C}$ in Remark \ref{Fronk}.

\smallskip

Secondly, we shall use the following version of the Bolzano-Weierstrass theorem.  
\begin{princ}[$\BW_{0}^{C}$]\label{comp}
For any countable $A\subset 2^{\N}$ and $F:2^{\N}\di 2^{\N}$, the supremum $\sup_{f\in A}F(f)$ exists. 
\end{princ}
Note that $\BW_{0}^{C}$ amounts to item \eqref{bokes3} for Cantor space and is provable\footnote{The \emph{intuitionistic fan functional} axiom $\MUC$ added to $\RCAo$ yields a conservative extension of $\WKL_{0}$ by \cite{kohlenbach2}*{Prop.\ 3.15} using $\ECF$.  Given a countable set $A\subset 2^{\N}$, let $Y:2^{\N}\di \N$ be injective on $A$ and apply $\MUC$ to obtain the upper bound of $Y$ on $2^{\N}$.  In this way, the set $A$ must be finite and the supremum from $\BW_{0}^{C}$ is now trivial to find (using $\MUC$).  Note that the previous proof does not really depend on how the set $A\subset 2^{\N}$ is given/represented/coded.} in a conservative extension of $\WKL_{0}$.
By contrast, over $\RCAo$, $\ACA_{0}$ is equivalent to the existence of $\sup_{n\in \N}F(f_{n})$ for \emph{sequences} $(f_{n})_{n\in \N}$ in $2^{\N}$ and \emph{any} $F:2^{\N}\di 2^{\N}$ (via the usual interval-halving proof).  
Nonetheless, $\BW_{0}^{C}$ is quite explosive when combined with the Suslin functional by the following theorem, a computational generalisation\footnote{The $\Sigma_{k+3}^{1}$-uniformisation theorems as in \cite{simpson2}*{VII.6.15} require certain set-theoretic assumptions, like the existence of $X\subset \N$ such that $(\forall Y\subset \N)(Y\in L(X))$, as defined in \cite{simpson2}*{VII.5.8}.  For this reason, we have not generalised Theorem \ref{BOOM} beyond $\SIX$.} of which may be found in Theorem \ref{BOOM2}.  
\begin{thm}\label{BOOM}
The system $\FIVE^{\omega}+\BW_{0}^{C}$ proves $\SIX$.  
\end{thm}  
\begin{proof}
The proof consists of two steps: we first show that $\BW_{0}^{C}\di \BOOT^{-}_{C}$ and then show that $\FIVE^{\omega}+\BOOT^{-}_{C}$ proves $\SIX$.
Here, $\BOOT^{-}_{C}$ is the statement that for all $Y^{2}$ such that $(\forall n^{0})(\exists \textup{ at most one } f\in 2^{\N})(Y(f, n)=0)$, we have
\be\label{boo}
(\exists X\subset \N)(\forall n\in \N)(n\in X\asa (\exists g\in 2^{\N})(Y(g, n)=0)).
\ee
Note that $\BOOT^{-}_{C}$ is essentially the restriction to $C$ of $\BOOT^{-}$.

\smallskip

First of all, let $Y^{2}$ be as in $\BOOT_{C}^{-}$.  
Define $G(w^{1^{*}}, k^{0})$ as $1$ if we have $(\exists i<|w|)(Y(w(i), k)=0)$, and $0$ otherwise.  
Define $f_{w}$ as $\lambda k.G(w, k)$ and define the set $A=\{ g\in 2^{\N} :(\exists k^{0})(Y(g, k)=0)\}$.
This set is countable via an obvious injection defined in terms of $Y$.
The set $B=\{w^{1^{*}}: (\forall i<|w|)(w(i)\in A)\}$ is similarly countable, as follows by considering $\r$ from Definition~\ref{keepintireal}.
Modulo coding and $\exists^{2}$, $B$ can be viewed as a subset of $2^{\N}$.  Define $F:2^{\N}\di 2^{\N}$ as $F(w):=f_{w}$ if $w\in B$, and $00\dots$ otherwise.  
Let $g$ be the supremum $\sup_{w\in B}F(w)$ and note that 
\[
(\forall n^{0})\big(g(n)=1\asa (\exists f\in 2^{\N})(Y(f, n)=0)\big),
\]
as required for $\BOOT^{-}_{C}$ and \eqref{boo}.

\smallskip

Secondly, to show that $\FIVE^{\omega}+\BOOT^{-}_{C}$ proves $\SIX$, we make use of the \emph{uniformisation result for $\Pi_{1}^{1}$-formulas}, provable in $\FIVE$ (see \cite{simpson2}*{VI.2}).
As noted in \cite{mummy}*{p.\ 530}, for a $\Sigma_{2}^{1}$-formula $(\exists X\subset \N)\psi(n, X)$ in $\L_{2}$, we may assume 
\be\label{wag}
(\forall n \in \N)(\exists \textup{ at most one } X\subset \N)\psi(n, X)
\ee
due to the aforementioned uniformisation result. 
Moreover, since $\psi(n, X)$ is $\Pi_{1}^{1}$, we may assume it has the normal form $(\forall g^{1})(\exists m^{0})(f(\overline{g}m, \overline{X}m, n)=0)$.  In particular, the latter formula is decidable given $\SS^{2}$ and 
\eqref{wag} becomes $(\forall n^{0})(\exists \textup{ at most one } X\subset \N)(Y( X, n)=0)$ where $Y(X, n)$ is $1-\SS(\lambda \sigma^{0^{*}}.f(\sigma, \overline{X}|\sigma|, n))$.
Applying $\BOOT^{-}_{C}$, the set $\{n\in \N: (\exists X\subset \N)(Y(X, n)=0)   \}$ is exactly $\{ n\in \N: (\exists X\subset \N)\psi(n, X)\}$, and $\SIX$ follows immediately. 
\end{proof}
We discuss the implications of the previous theorem in the following remark. 
\begin{rem}[Explosions and upper limits]\label{pimp}\rm
First of all, we have previously shown that the Lindel\"of lemma for Baire space yields $\FIVE$ when combined with $(\exists^{2})$ (\cite{dagsamIII, dagsamV}).  
We called this an `explosion' since the combination is quite strong compared to the components, which do not go beyond $\ACA_{0}$ \emph{in isolation}.  
Similar results exist for $\HBU$ and $(\exists^{2})$, which reach up to $\ATR_{0}$ in combination but are weak in isolation (\cite{dagsamVII}), namely not going beyond $\ACA_{0}$.
These results should be contrasted with `folklore' second-order results like that no true $\Pi_{2}^{1}$-sentence implies $\FIVE$, even given $\ATR_{0}$ (see \cite{ahar}*{Prop.\ 4.17}, which is titled `Folklore').

\smallskip

Secondly, $\FIVE^{\omega}$ is a conservative extension\footnote{The two final items of \cite{yamayamaharehare}*{Theorem 2.2} are (only) correct for $\QFAC$ replaced by $\QFAC^{0,1}$.} of $\FIVE$ for $\Pi_{3}^{1}$-formulas by \cite{yamayamaharehare}*{Theorem 2.2} and according to Rathjen in \cite{rathjenICM}*{\S3}, the strength of $\SIX$ \emph{dwarfs} that of $\FIVE$. 
Thus, Theorem \ref{BOOM} constitutes a new and more impressive explosion.  Moreover, $\SIX$ seems to be the current upper limit of RM, previously only reachable via rather abstract topology (\cites{mummy, mummymf, mummyphd}). 
By contrast, the Bolzano-Weierstrass theorem has \textbf{much} more of an `ordinary mathematics' flavour.

\smallskip

Thirdly, as to similar explosions, the topic of \cite{dagsamVII} is the study of the logical and computational theorems of ordinary mathematics \emph{pertaining to open sets}, where the latter are represented 
by characteristic functions.  
Now follows a list of such theorems that imply $\BOOT^{-}$, where `open' or `closed' set is to be interpreted as in \cite{dagsamVII}.  In light of its proof, Theorem~\ref{BOOM} then applies to the following theorems too. 
\begin{enumerate}
\renewcommand{\theenumi}{\alph{enumi}}
\item The perfect set theorem.
\item The Cantor-Bendixson theorem.
\item Any non-empty closed set $C\subseteq \R$ is located. 
\item $\open$: any non-empty open set in $\R$ is a union of basic open balls. 
\item The Urysohn lemma for $\R$.
\end{enumerate} 
In general, many theorems about open or closed sets as in \cite{dagsamVII} would imply the coding principle $\open$.  
Assuming $\NCC$ from Section \ref{uncountsyn2}, this also seems to be the case for the Tietze extension theorem, as discussed in \cite{dagsamIX}*{\S3}.
Moreover, we could restrict $\HAR_{0}$ to \emph{closed} sets and generalise $\WHBU$ to coverings of \emph{closed} sets; after this modification, the latter would imply the former. 
\end{rem}
Finally, we discuss why $\BW_{0}^{C}$ does indeed constitute a fragment of the `Bolzano-Weierstrass theorem for countable sets'.
We recall that subsets of $2^{\N}$ are studied in second-order RM via representations (\cite{simpson2}*{I.6.8}).  Thus, `$A_{0}\subset A_{1}$' is interpreted as $(\forall f\in 2^{\N})(\varphi_{A_{0}}(x)\di \varphi_{A_{1}}(x) )$ for certain $\varphi_{A_{i}}\in \L_{2}$ where we think of $A_{i}$ as being the set $\{f\in 2^{\N}: \varphi_{A_{i}}(f) \}$, although such sets $A_{i}$ are not part of the language $\L_{2}$.
\begin{rem}[On formalisation]\label{Fronk}\rm
When interpreted in $\ZFC$, $\BW_{0}^{C}$ expresses that $F(A):=\{g\in 2^{\N}:(\exists f\in A)(F(g)=A) \}$ has a supremum for countable $A\subset C$ and $F:C\di C$, where $F(A)$ is countable because $A$ is, Thus, $\BW_{0}^{C}$ clearly deals with the supremum of countable sets, namely $F(A)$, \emph{from the point of $\ZFC$}.    
It is a natural (and somewhat subtle) RM-question whether the latter insight is still valid when working in weak systems instead of $\ZFC$.  

\smallskip

Towards an answer, we first note that $F(A)$ as above does not always exist as a set in $\Z_{2}^{\omega}$.  
Indeed, the existence of $\{n\in \N: (\exists f\in 2^{\N})(Y(f)=n)  \}$ for any $Y^{2}$ is already equivalent to $\BOOT$ (\cite{samFLO2}).  
Hence, we cannot  hope to prove `$F(A)$ is a countable set' in $\Z_{2}^{\omega}$ because the latter does in general not even prove that $F(A)$ is a set, in the sense of being given by a characteristic function. 

\smallskip

Despite the previous, we can view $F(B)$ from the proof of Theorem \ref{BOOM} as a \emph{subset} of a countable set, where `inclusion' is interpreted in the `comparative' second-order sense mentioned just above.  
Indeed, one readily proves `$F(B)\subset D$', where $D=\{f \in 2^{\N} : (\exists n\in \N)(f=_{1}\sigma_{n}*00\dots) \}$ and where $\sigma_{n}$ is the $n$-th finite binary sequence.  
The set $D$ exists and is countable (following Definition \ref{standard}) given $\ACAo$.  To be absolutely clear, `$F(B)\subset D$' means the following:
\[
(\forall f\in 2^{\N})\big[ (\exists w\in B)(F(w)=_{1}f)  \di f\in D  \big].
\]
In light of the above, we can restrict $\BW_{0}^{C}$ to $A$ and $F$ such that $F(A)\subset D$ for some countable $D\subset 2^{\N}$.
This formulation is however far less elegant.   
\end{rem}

\subsection{The principle $\NBIJ$}\label{bereft}
In this section, we establish the results pertaining to $\NBI$ as summarised in Figure \ref{dd}.
Some of these results are variations of results about $\NIN$, while others are genuinely new.  

\subsubsection{The Axiom of Choice and $\NBI$}\label{uncountsyn2}
We connect $\NBI$ to choice principles, some of which provable in $\ZF$.
On a foundational note, we establish that $\Z_{2}^{\omega}+\neg\NBI$ is a rather strong (consistent) system in which $\R$ can be viewed as a potential infinity. 

\smallskip

First of all, given countable choice, there is the following obvious proof of $\NBIJ$.  
\begin{thm}\label{proleet}
The system $\RCAo+\QFAC^{0,1}$ proves $\NBIJ$, while $\Z_{2}^{\omega}$ does not.  
\end{thm}
\begin{proof}
For the first part, suppose $Y:[0,1]\di \N$ is a bijection.  Apply $\QFAC^{0,1}$ to $(\forall n\in \N)(\exists x\in [0,1])(Y(x)=n)$ to obtain a sequence $(x_{n})_{n\in \N}$ such that $Y(x_{n})=n$ for all $n\in \N$.  
Now use \cite{simpson2}*{II.4.9} to obtain $y\in [0,1]$ such that $y\ne_{\R} x_{n}$ for all $n\in \N$.  For $n_{0}=Y(y)$, we have $Y(x_{n_{0}})=Y(y)$, a contradiction. 

\smallskip

For the second part, we show that the model ${\bf Q}^*$ from Definition \ref{defQ} contains a bijection from ${\bf Q}[1]$ to $\N$. 
Using the notation from this definition, we construct a functional $F:A \rightarrow \N$ that will be both injective and surjective; we show that $F \in B_{0}$. 
Intuitively, $F$ is the limit of the increasing sequence of partial functionals $F_k:A_k \rightarrow \N$, while each $F_k$ is partially computable in $\SS^{2}_{k}$.  In this way, element-hood in $B_{0}$ is guaranteed, as required for the theorem. 
We define $F_k$ by recursion on $k$, and since $\exists^2$ is available, we may freely check equality between functions. We assume, in this proof, that the pairing function $\langle \cdot , \cdot \rangle$ is surjective.

\smallskip

Let $F_0(f) = \langle 0 , e\rangle$ where $e$ is  an index for computing $f$ from $\exists^2$, obtained using \emph{Gandy selection} as in Theorem \ref{gandyselection} and Corollary \ref{cor.inj}.
Now assume that $F_k$ is constructed such that if $F_k(f) = \langle i,d\rangle$ then $i \leq k$. Let $h_k \in A_{k+1}$ enumerate $A_k$. Let $\{h_{k,n}\}_{n \in \N}$ be a 1-1 enumeration of  all functions $h'$ that equal $h_k$ except for a finite set of arguments. Note that no such function $h'$ will be in $A_k$. Since the range of $F_k$ is computable in $\SS^2_{k+1}$, we can split the definition of $F_{k+1}$ in three cases (with case distinction decidable by $\SS^2_{k+1}$) as follows:
\begin{itemize}
\item[(i)] Put $F_{k+1}(f) = F_k(f)$ if $f \in A_k$.
\item[(ii)] Let $C_k$ be the (infinite by the definition of $F_{k}$) set of $c\in \N$ such that $\langle k,c \rangle$ is not in the range of $F_k$ and let $\{c_{k,n}\}_{n \in \N}$ enumerate $C_k$ in increasing order. Let $F_{k+1}(h_{k,n}) = \langle k , c_n \rangle$. Then $F_{k+1}$ fills in holes in the range left by $F_k$.
\item[(iii)] Suppose $f \in A_{k+1}$, $f \not \in A_k$, and $f \neq h_{k,n}$ for all $n$. 
Put $F_{k+1}(f) = \langle k+1 , e\rangle$, where $e$ is an index for computing $f$ from $\SS^2_{k+1}$ and $g_0 , \ldots , g_{k}$, obtained using Gandy selection (see Theorem \ref{gandyselection} and Corollary \ref{cor.inj}).
\end{itemize}
Item (ii) secures that the limit functional $F$ is surjective as well as injective.
\end{proof}
Since $\NIN$ implies $\NBIJ$, there is a proof of the latter in $\ZF$ \emph{and much weaker fragments}.  
Towards the latter kind of (elementary) proof, the following `weak' choice principle $\NCC$, provable in $\Z^{\Omega}_{2}$, was introduced in \cite{dagsamIX}.   
\begin{princ}[$\NCC$]
For $Y^{2}$ and $A(n, m)\equiv (\exists f\in 2^{\N})(Y(f, m, n)=0)$:
\be\label{garf}
(\forall n \in \N)(\exists m \in \N)A(n,m)\di  (\exists g:\N\di \N)(\forall n\in \N)A(n,g(n)). 
\ee
\end{princ}
The principle $\Sigma$-$\NFP$ is similar\footnote{Similar to \eqref{garf}, $\Sigma$-$\NFP$ states that for $A(\sigma^{0^{*}})\equiv (\exists g\in 2^{\N})(Y(g, \sigma)=0)$, we have 
\be\label{gru}
(\forall f \in \N^{\N})(\exists n \in \N)A(\overline{f}n)\di  (\exists g\in K_{0})(\forall f\in \N^{\N})A(\overline{f}g(f)),
\ee 
where `$g\in K_{0}$' means that $g$ is an RM-code.  The axiom $\NFP$ is \eqref{gru} for any formula $A$ and can be found in \cite{troeleke1}*{p.\ 215}; as studied in \cite{samph}, fragments of $\NFP$ populate the non-normal world.} to $\NCC$ and the former immediately implies the latter, and $\HBU$ together with $\WKL$, as shown in \cite{samph}*{\S5}.  To obtain $\HBC_{0}$, one applies $\Sigma$-$\NFP$ to the formula expressing that the countable set $A\subset \R^{2}$ from $\HBC_{0}$ provides a covering of $[0,1]$; $\WKL$ then implies that the resulting choice function has an upper bound on $[0,1]$.

\smallskip

The motivation for \cite{dagsamIX} and $\NCC$ was as follows: most results in \cites{dagsamIII, dagsamV, dagsamVI, dagsamVII, samph} that use $\QFAC^{0,1}$, 
go through with the latter replaced by $\NCC$; of course, $\Z_{2}^{\omega}$ does not prove $\NCC$.  The latter also does the job for $\NBIJ$ and Theorem \ref{proleet}, as follows.  
\begin{thm}\label{hungzo}
$\RCAo+\NCC$ proves $\NBIJ$, while $\Z_{2}^{\omega}+\NCC$ cannot prove $\NIN$.  
\end{thm}
\begin{proof}
The second part is immediate by Theorem \ref{dick}.
For the first part, let $Y:[0,1]\di \N$ be a bijection and note that we have $(\exists^{2})$ due to \cite{kohlenbach2}*{\S3}.  
The functional $\exists^{2}$ allows us to convert real numbers in $[0,1 ]$ to binary representation.  
Consider the following (trivial) formula
\be\label{comf}
(\forall n\in \N)(\exists q\in \Q\cap [0,1])\big[ (\exists x\in [0,1])(Y(x)=n \wedge [x](2n+4)=q   \big].
\ee
Modulo obvious coding, the square-bracketed formula in \eqref{comf} has the right form for applying $\NCC$.  
Let $(r_{n})_{n\in \N}$ be the resulting sequence of rationals.  
We now consider the proof of \cite{simpson2}*{II.4.9}.  The latter expresses that for every sequence $(x_{n})_{n\in \N}$ of real numbers, there is $y\in [0,1]$ such that $y\ne x_{n}$ for all $n\in \N$.
The real $y$ is defined as the limit $\lim_{n\di \infty}a_{n}$ where $(a_{0}, b_{0})=(0,1)$, $[x_{n}]=q_{n, k}\in \Q$, and
\be\label{enuff}\textstyle
(a_{n+1}, b_{n+1}):=
\begin{cases}
(\frac{a_{n}+3b_{n}}{4} , b_{n}) & q_{n, 2n+3} \leq (a_{n}+b_{n})/2 \\ 
(a_{n }, \frac{3a_{n}+b_{n}}{4}) & \textup{otherwise}
\end{cases}.
\ee
The crux now is to observe that in \eqref{enuff}, one only uses \emph{finitely much information} about each $x_{n}$, namely the approximation $q_{n, 2n+3}$.   
Hence, using $r_{n}$ instead of $q_{n, 2n+3}$, \eqref{enuff} provides a real $y\in [0,1]$ which is such that if $Y(x)=n$, then $x\ne_{\R} y$, for any $x\in [0,1]$ and $n\in \N$.  
Now for $n_{0}:= Y(y)$, there is $z\in [0,1]$ such that $Y(z)=n_{0}$, but $z\ne_{\R} y$ by construction, a contradiction. 
\end{proof}
The previous proof is quite illustrative: $\QFAC^{0,1}$ is often used to produce a sequence of reals; if the subsequent argument only needs
finitely much information about each element in the sequence (which is often the case), then $\NCC$ suffices.   
Let $\NCC_{\w}$ be $\NCC$ with $g$ in \eqref{garf} only providing an \emph{upper bound} on the $m$-variable.
\begin{cor}
The theorem remains valid if we replace $\NCC$ by $\NCC_{\w}$.
\end{cor}
\begin{proof}
We assume $\neg\NBI$ and therefore have $(\exists^{2})$.  The latter allows us to find a bijection $F:\N^{\N}\di \N$.  Now consider the following formula $A(n, m)$:
\[\textstyle
(\exists w^{1^{*}})\big[|w|=n+1\wedge  (\forall i<|w|)(F(w(i)) = i-1) \wedge m =( \sum_{i = 1}^n w(i)(n)) + 1],
\]
and note that $(\forall n\in \N)(\exists m\in \N)A(n, m)$.  Let $g:\N\di \N$ be such that $(\forall n\in \N)(\exists m\leq g(n))A(n, m)$ as provided by $\NCC_{\w}$ (modulo obvious coding).
By definition, the function $g$ dominates all other functions, which yields a contradiction. 
\end{proof}
\subsubsection{Countable sets versus sets that are countable II}\label{uncountsym2}
In this section, we derive $\NBI$ from basic theorems about \emph{countable sets} where the latter has its \emph{usual} meaning, namely Definition~\ref{standard2} taken from \cite{hrbacekjech}.
Corollary~\ref{holycoroly} suggests an elegant base theory in which (theorems about) strongly countable sets behave `as they should'.  

\smallskip

First of all, we generalise Corollary \ref{forgot} to a stronger definition of `countable set', as can e.g.\ be found in \cite{hrbacekjech}.
We note that Borel uses this definition in \cite{opborrelen2}.
\bdefi[Strongly countable]\label{standard2}~
A set $A\subseteq \R$ is \emph{strongly countable} if there is $Y:\R\di \N$ with $(\forall x, y\in A)(Y(x)=Y(y)\di x=y)\wedge(\forall n\in \N)(\exists x\in A)(Y(x)=n)$. 
\edefi
Let $\HBC_{1}$ be $\HBC_{0}$ from Section \ref{uncountsyn} restricted to strongly countable sets.  
\begin{thm}
$\RCAo$ proves $\HBC_{1}\di \NBI$ while $\Z_{2}^{\omega}+\HBC_{1}$ cannot prove $\NIN$.
\end{thm}
\begin{proof}
The negative result follows from Theorem \ref{dick} as $\QFAC^{0,1}$ yields a sequence enumerating a strongly countable set, i.e.\ $\Z_{2}^{\omega}+\QFAC^{0,1}$ proves $\HBC_{1}$.  
For the positive result, let $Y:[0, 1]\di \N$ be a bijection and note that $[0,1]$ is strongly countable as in Definition \ref{standard2}.
Fix $\Psi:[0,1]\di \R^{+}$ and define $A$ as the strongly countable set $\{ (x-\Psi(x), x+\Psi(x)):x\in [0,1] \}$.   
Applying $\HBC_{1}$, there is a finite sub-cover and $\HBU$ follows.  Since $\HBU\di \WHBU$, Theorem \ref{flahu} yields $\NIN$, which contradicts our assumption $\neg \NBI$. 
\end{proof}
In the same way as in the previous proof, items \eqref{bokes}-\eqref{bokes3} from Section \ref{uncountsyn} formulated with Definition \ref{standard2} all yield principles that imply $\NBI$ but not $\NIN$.
This shall follow from the below results pertaining to $\DCA$ from Principle \ref{DAAS}.

\smallskip

Next, it is a natural question what the weakest comprehension axiom is that still proves $\NBI$.  As it happens,
we have a candidate, namely \emph{$\Delta$-comprehension} as in Principle \ref{DAAS}.
We discuss the importance of $\DCA$ in Remark \ref{DCA}.    
\begin{princ}[$\DCA$]\label{DAAS}
For $i=0, 1$, $Y_{i}^{2}$, and $A_i(n)\equiv (\exists f \in \N^\N)(Y_{i}(f,n)=0)$\textup{:}
\[
(\forall n\in \N)(A_0(n) \asa \neg A_1(n))\di (\exists X\subset \N)(\forall n\in \N)(n\in X\asa A_{0}(n)).
\]
\end{princ}
\noindent
As shown in \cite{dagsamIX}*{\S3.1}, $\NCC$ implies $\DCA$, but the following proof is interesting.
\begin{thm}\label{exnihilo}
$\RCAo$ proves $\DCA\di \NBI$, while $\Z_{2}^{\omega}+\DCA$ cannot prove $\NIN$.
\end{thm}
\begin{proof}
For the second part, $\RCAo+\QFAC^{0,1}$ proves $\DCA$ by \cite{samFLO2}*{Theorem 3.5}.  
This is done by applying $\QFAC^{0,1}$ to $(\forall n\in \N)(\neg A_{1}(n)\di  A_{0}(n))$.
Hence, $\Z_{2}^{\omega}+\DCA$ cannot prove $\NIN$ by (the second part of) Theorem \ref{dick}. 
Alternatively, $\NCC\di \DCA$ by \cite{dagsamIX}*{Theorem 3.1} and use Theorem \ref{hungzo}.  

\smallskip

For the first part, assume $\DCA$ and let $Y:[0,1]\di \N$ be a bijection.  
Recall we have access to $(\exists^{2})$. 
Let $(q_{n})_{n\in \N}$ be an enumeration of the rationals in $[0,1]$.
Consider the following equivalence for $n, m\in \N$ and $r\in \Q^{+}$:  
\be\label{frogke}
(\exists x\in B(q_{m}, r)\cap [0,1])(Y(x)=n)\asa (\forall y\in [0,1]\setminus B(q_{m}, r))(Y(y)\ne n).
\ee
Note that the reverse implication in \eqref{frogke} only holds because $Y$ is a bijection.  Modulo coding, $\DCA$ yields $X\subset \N^{2}\times \Q$ such that for all $n, m\in \N$ and $r\in \Q^{+}$:
\be\label{xdg}
(n,m, r)\in X\asa (\exists x\in B(q_{m}, r)\cap [0,1])(Y(x)=n)
\ee
We now use $X$ and the well-known interval-halving technique to create a sequence $(x_{n})_{n\in \N}$ as follows.  
For fixed $n\in \N$, define $[x_{n}](0)$ as $0$ if $(\exists x\in [0, 1/2))(Y(x)=n)$, and $1/2$ otherwise; define $[x_{n}](m+1)$ as $[x_{n}](m)$ if $(\exists x\in \big[[x_{n}](m), [x_{n}](m)+\frac{1}{2^{m+1}}\big)\big)(Y(x)=n)$, and $[x_{n}](m)+\frac{1}{2^{m+1}}$ otherwise.  By definition, we have $Y(x_{n})=n$ for all $n\in \N$; use \cite{simpson2}*{II.4.9} to find $y_{0}\in [0,1]$ not in the sequence $(x_{n})_{n\in \N}$.
Then $n_{0}=Y(y_{0})$ yields a contradiction as $x_{n_{0}}\ne_{\R} y $ and $Y(y_{0})=Y(x_{n_{0}})$.
\end{proof}
The previous proof inspired us to formulate Example \ref{XXY} in Appendix \ref{extraextrareadallaboutit}.
For the next corollary, let $\cocode_{1}$ be $\cocode_{0}$ restricted to strongly countable sets.  
\begin{cor}\label{holycoroly}
$\RCAo$ proves $\DCA\di \cocode_{1}$ and $[\cocode_{1}+\WKL]\di \HBC_{1}$.
\end{cor}
\begin{proof}
In the final part of the proof of the theorem, replace `$[0,1]$' by `$A$' and note that $\DCA\di \cocode_{1}$ follows.  
The second implication is by \cite{simpson2}*{IV.1.1}. 
\end{proof}
As argued in Remark \ref{DCA} and \cites{samrecount, samFLO2}, $\DCA$ yields a good base theory 
\emph{for a purpose rather unrelated to our current enterprise}.  
Another argument in favour of this axiom is that $\RCAo+\DCA$ proves $\WKL\di \HBC_{1}$ by Corollary \ref{holycoroly}, a very desirable \emph{catharsis} in light of Theorem~\ref{proleet}.  Similarly, $\ACA_{0}\di \BW_{1}$ given $\DCA$, where the former is $\BW_{0}$ restricted to strongly countable sets.     
\begin{rem}[Lifting proofs]\label{DCA}\rm
As suggested by its structure, $\DCA$ is the `higher-order' version of $\Delta_{1}^{0}$-comprehension, where the latter is included in $\RCA_{0}$.  
Using $\DCA$, one can almost verbatim `lift' second-order proofs to more general and interesting proofs in third-order (and higher) arithmetic.
As an example, we consider the proof that the \emph{monotone convergence theorem} implies $\ACA_{0}$ from \cite{simpson2}*{III.2.2} based on \emph{Specker sequences}.  
As explored in detail in \cites{samph, samFLO2, samrecount}, one can use this same proof \emph{with no essential modification} to establish that 
the \emph{monotone convergence theorem for \textbf{nets}} implies $\BOOT$ based on \emph{Specker \textbf{nets}}.
In this `lifted' proof, one uses $\DCA$ instead of $\Delta_{1}^{0}$-comprehension.  A proof of this implication \emph{not} using $\DCA$ is also given in \cite{samph}*{\S3}, but the point is that 
second-order proofs can be `recycled' as interesting proofs in third-order (and higher) arithmetic.  Many examples are discussed in \cite{samrecount, samFLO2}. 
\end{rem}
We suspect a connection between $\DCA$ and $\Delta_{1}^{1}$-comprehension (see \cite{simpson2} for the latter), but no evidence can be offered at this point in time.  

\smallskip

Finally, we note that $\HAR_{1}$, i.e.\  $\HAR_{0}$ restricted to strongly countable sets, satisfies $\cocode_{1}\di \HAR_{1}\di \NBI$.
In the same way as for Corollary~\ref{frex}, $\Z_{2}^{\omega}$ cannot prove that strongly countable RM-closed sets are given by a sequence.  
\begin{cor}\label{frex2}
The restriction of $\cocode_{1}$ or $\HAR_{1}$ to RM-closed sets $A\subset \R$ still implies $\NBI$.
\end{cor}
One can push the previous corollary further as follows: $\NBI$ follows from the statement \emph{an RM-closed strongly countable subset of $\R$ has measure $<+\infty$.}

\smallskip

In conclusion, we note that the above results (mainly) pertain to analysis, but $\NBI$ even follows from the {original} \emph{graph-theoretical} lemma by K\"onig from \cite{koning147}, as discussed in Example \ref{takethat}.  

\section{The uncountability of $\R$ in computability theory}\label{floep}
We establish the computational properties of $\NIN$ following Kleene's framework introduced in Section \ref{HCT}.  
\subsection{Comprehension and fan functionals}
We show that $N$ as in $\NIN(N)$ is \emph{hard to compute}, relative to the usual scale of comprehension functionals $\SS_{k}^{2}$.
\be\tag{$\NIN(N)$}
(\forall Y:[0,1]\di \N)(  N(Y)(0)\ne_{\R} N(Y)(1)\wedge  Y(N(Y)(0))=Y(N(Y)(1)) ).
\ee
We again stress that this should be interpreted as support for the study of non-normal functionals of type~3, as in e.g.\ Corollary~\ref{corkei}.
In fact, we shall study \emph{strongly} non-normal functionals, i.e.\ functionals that do not compute $\exists^{3}$ \emph{even relative to $\exists^{2}$}, as defined in Section \ref{floep2}.
The type of $N$ as in $\NIN(N)$ is written `$3$' for simplicity. 
\begin{thm}\label{ninot}
A functional $N^{3}$ as in $\NIN(N)$ is not computable from any type two functional.
\end{thm}
\begin{proof}Let $F$ be of type 2; without loss of generality we may assume that $F$ is normal.  Let $G$ be the partial functional obtained from $F$ as in Corollary  \ref{cor.inj}. Let $H$ be any total extension of $G$. If $N$ as in $\NIN(N)$ is computable in $F$, then $N(H) = (x,y)$ such that $H(x) = H(y)$ for $x\ne y$. However, by Lemma \ref{lemma.blowout} both $x$ and $y$ are computable in $F$, so $H(x)\ne H(y)$ by the choice of $G$.
\end{proof}
While the proofs in the previous section are by contradiction, we now show that a realiser for $\WHBU$ also computes $N$ as in $\NIN(N)$.
As to the former, a \emph{$\Lambda$-functional} (or: \emph{weak fan functional}) is a type three functional such that $\Lambda(Y, \eps)$ outputs the finite sequence $y_{0}, \dots, y_{k}$ of distinct reals as in $\WHBU$.  
Slightly different definitions are used in e.g.\ \cites{dagsam, dagsamII} for `realisers for $\WHBU$', but all are equivalent up to a term of G\"odel's $T$.  The following is officially a corollary to Theorem \ref{flahu}. 
\begin{cor}\label{corkei}
Any $\Lambda$-functional computes $N$ as in $\NIN(N)$.  
\end{cor}
\begin{proof}
To define $N(Y)$, define $\Psi_{0}(x):=\frac{1}{2^{Y(x)+3}}$ and consider $\Lambda(\Psi_{0}, \frac{1}{2})=(y_{0}, \dots, y_{k})$, where the $y_{i}$ are assumed distinct. 
Note that since $\frac{1}{2}<_{\R}\sum_{i\leq k}| J_{y_{i}}^{\Psi_{0}}|$ by definition, we cannot have that all $Y(y_{i})$ are distinct, similar to the proof of Theorem \ref{flahu}.  
Thus, let $N(Y)$ output any two $y_{i}, y_{j}$ in $\Lambda(\Psi_{0}, \frac{1}{2})$ such that $Y(y_{i})=Y(y_{j})$.
\end{proof}
Next, we show that realisers for $\BCT$ cannot be computed by any type-two functional, based on (the proof of) Theorem \ref{flahu} and Lemma \ref{lemma.blowout}.
This yields a simpler proof of \cite[Theorem 6.6]{dagsamVII}, which has the same content as Theorem \ref{triplez}. 
We recall that open sets (here and in \cite{dagsamVII}) are given by characteristic functions.
\bdefi[Realiser for $\BCT$]
A \emph{Baire-realiser} is a total functional $\zeta$ that takes as input a sequence $ \{X_n : n \in \N\}$ of subsets of $[0,1]$, and outputs a real $\zeta(\{X_n : n \in \N\}) \in \bigcap_{n \in \N}X_n$ whenever each $X_n$ is open and dense.
\edefi
\begin{thm}\label{triplez}
No Baire realiser is computable in any type two functional.  
\end{thm}
\begin{proof}
We take the (computational) connection between $[0,1]$ and $2^\N$ to be known.
Let $F$ be a normal functional of type 2, and assume that a Baire realiser $\zeta$ is computable in $F$. By Corollary \ref{cor.inj} there is a partial and injective functional $G$ with integer values and defined on all reals in $[0,1]$ computable in $F$ . Let $x \in X_n$ if $x$ is not computable in $F$, or if $x$ is computable in $F$ and $G(x) > n$. Each $X_n$ is open and dense, and the sequence $\{X_n : n \in \N\}$ is partially computable in $F$ on the set of reals computable in $F$. By Lemma \ref{lemma.blowout} and the assumption on $\zeta$, we must have that $\zeta(\{X_n : n \in \N\})$ is computable in $F$, contradicting the fact that $\cap_{n}X_n$ contains no reals computable in $F$.
\end{proof}
Finally, we obtain a computational generalisation of Theorem \ref{BOOM}, where item~\eqref{fli2} in Theorem \ref{BOOM2} corresponds to the former theorem.  
\begin{definition} 
Let $\Omega_{\BW}(A^2,Y^2,F^2) = \sup\{F(f) : f \in A\}$ whenever $A \subseteq 2^{\N}$, $Y:2^{\N} \rightarrow \N$ is total on $2^{\N}$ and injective on $A$, and $F:2^{\N} \rightarrow 2^{\N}$.
\end{definition}
By definition, $\Omega_{\BW}$ is a fixed partial object of type 3 that is not countably based, but with some surprising computational properties. 
\begin{theorem}\label{BOOM2}~
\begin{enumerate}
 \renewcommand{\theenumi}{\alph{enumi}}
\item If $f:\N\di \N$ is computable in $\Omega_{\BW}$ and $\exists^2$, then $f$ is hyperarithmetical.\label{fli1}
\item The functional $\SS^2_2$ is computable in $\Omega_{\BW}$ and the Suslin functional $\SS^2$.\label{fli2}
\item If \textsf{\textup{V = L}}, then $\exists^3$ is computable in $\Omega_{\BW}$ and the Suslin functional $\SS^2$.\label{fli3}
\end{enumerate}
\end{theorem}
\begin{proof}
For item \eqref{fli1}, let $A$, $Y$, $F$ be computable in $\exists^2$, and such that $\Omega_{\BW}(A,Y,F)$ is defined. If $g \in A$, let $n = Y(g)$. Then $\{g\} = \{f \in C : f \in A \wedge Y(f) = n\}$, so $g$ is hyperarithmetical. 
Using Gandy selection and the boundedness theorem for computations relative to $\exists^2$, we can find an enumeration of $A$ computable in $\exists^2$ uniformly computable in the indices for $A$ and  $Y$.  From this, we obtain an index for the hyperarithmetical least upper bound of $\{F(f) : f \in A\}$. The recursion theorem (relative to $\exists^2$) then yields a primitive recursive function $\rho$ such that  
\[
\{e\}(\Omega_{\BW},\exists^2,\vec f , \vec a) = b \rightarrow \{\rho(e)\}(\exists^2,\vec f , \vec a) = b.
\]
For item \eqref{fli2}, 
let $B\subset \N$ be given as follows: $n \in B\asa ( \exists f \in 2^{\N})P(f,n)$ where $P \in \Pi^1_1$ (and in $\L_{2}$) and where for all $n$ there is at most one $f\in 2^{\N}$ with $P(f,n)$. 
This is a normal form for $\Sigma^1_2$ thanks to the well-known $\Pi_{1}^{1}$-uniformisation theorem.
Now define the set $A\subset 2^{\N}$ as follows: $f \in A$ if $(\exists n\in \N) P(f,n)$ and define $G(f) = \{n : P(f,n)\}$.
The set $A$ is countable which can be see by selecting the least $n\in \N$ with $P(f,n)$ if there is one, and $0$ otherwise.   
This injection on $A$ is computable in $\SS^{2}$.
Using $\Omega_{\BW}$, we obtain $B = \cup\{G(f) : f \in A\}$, as required.

\smallskip

For item \eqref{fli3},  let $Y$ be the constant 1 function and let $F$ be the identity function on $2^{\N}$.
Now assume that \textsf{V = L}. Since we also assume the Suslin functional $\SS^{2}$ (and hence $\exists^{2}$), it suffices to compute $\kappa_{0}^3$ from \cite{dagsam} defined as
\be\label{tochk}
(\forall Y^{2})\big[  \kappa_{0}(Y)=0\asa (\exists f\in 2^{\N})(Y(f)=0  ) \big].
\ee
Given $g \in 2^{\N}$, the following relation is $\Pi_{1}^{1}$ :
\be\label{ziedewel}
\textup{$f$ is a code for an initial segment $(L_\alpha,<_\alpha)$ of $(L,<)$ with $g \in L_\alpha$}.
\ee
By the proof of item \eqref{fli2}, i.e.\ using $\Pi^1_1$-uniformisation, there is a function $H^2:2^{\N} \rightarrow \N$ computable in $\SS^{2}$ and $\Omega_{\BW}$ such that $H(g)$ is a code $f$ as in \eqref{ziedewel}.
Now, given $Z \subseteq C$, we let $Z^*$ be the set of `minimal' $g \in Z$, in the sense that for $h < g$ in the well-ordering of $L$, we have that $h \not \in Z$. 
Then $Z^*$ consists of at most one element and is arithmetically definable using $Z$ and $H$.  Applying $\Omega_{\BW}$  to $(Z^*,Y,F)$ yields:
\[
(\exists g \in 2^{\N} )(g \in Z) \leftrightarrow \Omega_{\BW}(Z^*,Y,F) \in Z,
\]
which gives us $\kappa_{0}^{3}$ as in \eqref{tochk}, and hence $\exists^{3}$. 
\end{proof}
While item \eqref{fli2} shows that $\Omega_{\BW}$ is rather powerful when combined with $\SS^{2}$, item~\eqref{fli1} shows that $\Omega_{\BW}$ is rather tame in the presence of $\exists^{2}$, as $f:\N\di \N$ is hyperarithmetical if and only if it is computable from $\exists^{2}$.  
This leads to the following corollary, where a \emph{Pincherle realiser} (PR for short) is any functional that outputs an upper bound on the length of the finite sub-cover from $\HBU$.  
A detailed study of PRs may be found in \cite{dagsamV}.
\begin{cor}
No PR can be computable in $\Omega_{\BW}+\exists^{2}$.
\end{cor}
\begin{proof}
By \cite{dagsamV}*{Cor.\ 3.8}, the combination of any PR and $\mu^{2}$ can compute functions $f:\N\di \N$ that are not hyperarithmetical.   
\end{proof}
Similar to the previous corollary, we believe that $\Z_{2}^{\omega}+\BW_{0}^{C}$ cannot prove $\HBU$, but do not have a proof at the moment. 

\smallskip

Finally, we should mention Hartley's results \cite{hartleycountable} where it is shown that, assuming $\textsf{CH}$, a functional of type 3 that is not countably based will compute $\exists^3$ relative to some functional $F$ of type 2. 
Surprisingly, in case that \textsf{V = L} holds and $\Omega_{BW}$ is given, we may chose the Suslin functional $\SS^{2}$ for this functional $F^{2}$.

\subsection{Computing \emph{de dicto} and \emph{de re}}
In this section, we discuss some subtle variations of the concept of `realiser for open-cover compactness', and how this `trickles down' to realisers for $\NIN$.

\smallskip

Now, the counterpart of $\Lambda$-functionals for $\HBU$ are called $\Theta$-functionals, i.e.\ realisers for $\HBU$ that return the finite sub-cover from the latter (see e.g.\ \cite{dagsam, dagsamII, dagsamIII, dagsamV}). 
Closely related, a \emph{Pincherle realiser} (PR for short) is a functional $M^{3}$ that returns an upper bound $M(\Psi)$ on the \emph{length} of finite sub-covers from $\HBU$ (see e.g.\ \cite{dagsamV}).
Hence, $\Theta$-functionals provide some finite sub-cover, while PRs provide a natural number (only) such that a finite sub-cover of this length (or shorter) \emph{exists}.  
In this spirit, we define a weak variation of $\NIN(N)$ as follows.
\begin{princ}[$\NIN_{0}(N_{0})$]
\[
(\forall Y:[0,1]\di \N)(\exists x, y\in [0,1])( x\ne y\wedge   Y(x)=Y(y)\wedge Y(x)\leq N_{0}(Y) ).
\]
\end{princ}
Similar to a PR, $N_{0}(Y)$ satisfying $\NIN_{0}(N_{0})$ does not return two real numbers that map to the same natural number, but only an upper bound for the latter.
We still have the following property. 
\begin{thm}\label{Xx}
A functional $N_{0}^{3}$ as in $\NIN_{0}(N_{0})$ is not computable from any type two functional.
\end{thm}
\begin{proof}
We modify the proof of Theorem \ref{ninot}.  Let $F$ and $G$ be as in the proof of the latter. Let $H_n$ be the extension of $G$ that is constant $n$ outside the domain of $G$. If $\NIN_{0}(N_0)$ we must have that $N_0(H_n) \geq n$. On the other hand, if $N_0$ is computable in $F$ we must have that $N_0(H_n)$ is independent of $n$, by Lemma \ref{lemma.absolute}. 
Thus, $N_0$ is not computable in $F$ and we are done.
\end{proof}
A number of `weak' functionals do compute $N_{0}$ as in $\NIN_{0}(N_{0})$.  
It is interesting to note that even very weak statements of measure theory yield functionals that are hard to compute as in Theorem \ref{Xx}.  
Indeed, recall Tao's pigeon hole principle $\PHM$ from Section \ref{uncountsyn} and let $\PHM(T^{3})$ be the statement that for a sequence $(E_{n})_{n\in \N}$ of sets in $[0,1]$, $T(\lambda n.E_{n})=n_{0}$ is such that $E_{n_{0}}$ has positive measure if the union $\cup_{n\in \N}E_{n}$ has positive measure and is RM-closed (see Corollary \ref{frex}). 
\begin{thm}
Any $T^{3}$ as in $\PHM(T)$ computes $N_{0}$ as in $\NIN_{0}(N_{0})$.
Any PR computes $N_{0}$ as in $\NIN_{0}(N_{0})$.
\end{thm}
\begin{proof}
For the second part, fix $Y:[0,1]\di \N$, let $M$ be a PR, and consider $k_{0}:=M(\Psi_{0})$ for $\Psi_{0}(x):= \frac{1}{2^{Y(x)+2}}$. 
By the definition of PR, there are distinct $y_{0}, \dots, y_{k_{0}}\in [0,1]$ such that $\cup_{i\leq k_{0}}I_{y_{i}}^{\Psi_{0}}$ covers $[0,1]$.  
In particular, we have $1<_{\R}\sum_{i\leq k}| J_{y_{i}}^{\Psi_{0}}|$.
However, if $Y(y_{i})\ne Y(y_{j})$ for all $i, j\leq k_{0}$, then $\frac{1}{2}>_{\R}\sum_{i\leq k}| J_{y_{i}}^{\Psi_{0}}|$ by the definition of $\Psi_{0}$.
Hence, for some $i, j\leq k_{0}$ we must have $Y(y_{i})=Y(y_{j})$ and define $N_{0}(Y):=M(\Psi_{0})$, which satisfies $\NIN_{0}(N_{0})$

\smallskip

For the first part, let $T^{3}$ satisfy $\PHM(T)$. 
Fix $Y:[0,1]\di \N$ and define $E_{n}:= \{ x\in [0,1]:Y(x)=n  \}$.  
Clearly $[0,1]=\cup_{n\in \N}E_{n}$ has positive measure and let $T(\lambda n. E_{n})=n_{0}\in \N$ be such that $E_{n_{0}}$ has positive measure. 
There must be at least two reals in $E_{n_{0}}$ as the empty set and singletons have measure zero by definition.  
Define $N_{0}(Y)$ as this number $n_{0}$ and note that $\NIN_{0}(N_{0})$.
\end{proof}
Note that the previous proof still goes through if we require that the coverings from the definition of `$\cup_{n\in \N}E_{n}$ has positive measure' are given as input for $T^{3}$.  
Moreover, combined with $\mu^{2}$, the functional $\blambda$ as in Example \ref{frag} computes $N_{0}$ as in $\NIN_{0}(N_{0})$ in the same way as realisers for $\PHM$ do. 

\smallskip

In light of the previous, we offer the following conjecture. 
\begin{conj}
No PR can compute any $N$ as in $\NIN(N)$.  No $T^{3}$ as in $\PHM(T)$ can compute any $N$ as in $\NIN(N)$.
\end{conj}
Finally, we have previously discussed inductive definitions (\cite{dagcie18, dagsamVII}) and 
the following amusing observation illustrates the power of non-monotone induction; it brings us nothing new with respect to known relative computability.
Fix some $F:\N^{\N}\di \N$ and define $G_F(A) := A \cup {F(A)}$ for $A\in 2^{\N}$.  Non-monotone inductive definitions yield the existence of $A$ such that $G_{F}(A)\subseteq A$.  
In fact, this induction stops exactly when we have found $A \ne B$ such that $F(A) = F(B)$, so this induction is thus a simple $\NIN$-realiser.

\appendix
\section{Principles related to the uncountability of $\R$}\label{extraextrareadallaboutit}
In this appendix, we list some results related to $\NIN$ and $\NBI$.
We only sketch the results as introducing the extra technical machinery (say in $\RCAo$) would be cumbersome or take too much space. 
All but the first result are positive in nature.
\begin{exa}[Well-ordering the reals]\rm
Assuming sub-sets of $[0,1]$ are given as characteristic functions,  $\Z_{2}^{\omega}+\QFAC^{0,1}+\WO([0,1])$ does not imply $\NIN$, where $\WO([0,1])$ expresses that \emph{the unit interval can be well-ordered}.
Indeed, $\neg\NIN$ readily implies $\WO([0,1])$ by noting that $x\preceq y \equiv Y(x)\leq_{0}Y(y)$ yields a well-order in case $Y:[0,1]\di \N$ is an injection. 
Similarly, the latter observation establishes that $\neg\NBI$ is equivalent to the statement \emph{there is a total order $\trianglelefteq$ of $[0,1]$ such that $([0,1], \trianglelefteq)$ is order-isomorphic to $(\N, \leq_{\N})$}.
\end{exa}
\begin{exa}[Ramsey's theorem]\rm
It is well-known that (infinite) Ramsey's theorem for two colours and pairs, abbreviated $\RT_{2}^{2}$, does not generalise beyond the countable.  
This failure is denoted `$2^{\aleph_{0}}\not\di (2^{\aleph_{0}})_{2}^{2}$' and can be found in \cite{reim}*{Prop.~2.36}, going back to Sierpi\'nski (\cite{grotesier2}).  
Assuming that $\R$ has a total order $\preceq$ in which each $a\in \R$ has a unique successor $S(a)\in \R$, one can use the aforementioned proof by Sierpi\'nski to show that $2^{\aleph_{0}}\di (2^{\aleph_{0}})_{2}^{2}$ implies $\neg\NIN$. 
\end{exa}
\begin{exa}[Cantor-Schr\"oder-Bernstein theorem]\label{CSB}\rm
An early theorem of set theory that implies $\NIN$ when combined with $\NBI$ is the \emph{Cantor-Schr\"oder-Bernstein theorem}, originally published without proof by Cantor in \cite{cantor3} (see \cite{cantor33}*{p.~413}).
This theorem states that if there is an injection $f:A\di B$ and an injection $g:B\di A$, then there is a bijection between $A$ and $B$.
Thus, assuming $\neg \NIN$, there would be a bijection between $[0,1]$ and $\N$, contradicting $\NBI$.
\end{exa}
\begin{exa}[Perfect sets]\rm
Cantor proves the following in \cite{cantorb}*{\S16} around 1879:
\[
\textup{\emph{If a subset $A\subset \R^{n}$ is countable, then it cannot be perfect}}.
\]
The restriction of this theorem to $[0,1]$ (rather than $\R^{n}$) readily implies $\NIN$, where `perfect' means `closed without isolated points', like in RM (\cite{simpson2}*{VI.1.4}).
\end{exa}
\begin{exa}[RM of topology]\label{mummyrem}\rm
The RM of topology is developed in e.g.\ \cite{mummy, mummyphd, mummymf}, working in second-order arithmetic.  
Topological spaces are represented via \emph{countable bases} and Hunter has investigated the existence of the 
latter in higher-order RM (\cite{hunterphd}), with some striking results.  Indeed, countable bases are intimately connected to $(\exists^{3})$ by \cite{hunterphd}*{Prop.\ 2.15}.  Our results are more modest, but significant nonetheless: countable bases in second-order RM are given by a sequences.  Hence, one seems to need $\cocode_{0}$ (or $\cocode_{1}$) to guarantee that the scope of the second-order RM of topology is the same as the RM of topology for (strongly) countable bases \emph{when formulated with Definition \ref{standard}}, i.e.\ as usual.
\end{exa}
\begin{exa}[Baire category theorem]\rm
We have studied the connection between $\NIN$ and $\BCT$ in Theorem \ref{flahu}.  One can also formulate $\BCT'$ which states that for a \emph{countable collection} of dense RM-open sets in $\R$, 
there is at least one real in all the members of this collection; one readily proves that $\BCT'\di \NIN$.
We note that Baire used terms like `infinit\'e d\'enombrable d'ensembles' (=countable infinity of sets) in the formulation of (what we now call) the Baire category theorem (see \cite{beren2}*{p.\ 65}).
In this way, $\BCT'$ is actually quite close to the historical original. 
\end{exa}
\begin{exa}[Baire classes]\rm
One can derive $\NIN$ from basic properties of \emph{Baire classes} on the unit interval.
Now, \emph{Baire classes} go back to Baire's 1899 dissertation (\cite{beren}).
A function is `Baire class $0$' if it is continuous and `Baire class $n+1$' if it is the pointwise limit of Baire class $n$ functions.  
Each of these levels is non-trivial and there are functions that do not belong to any level, as shown by Lebesgue (see \cite{kleine}*{\S6.10}). 
Baire's \emph{characterisation theorem} (\cite{beren}*{p.\ 127}) expresses that a function is Baire class $1$ iff there is a point of continuity of the induced function on each perfect set.
Using the latter formulation of Baire class 1, $\NIN$ follows from either of the statements \emph{Baire class $2$ does not contain all functions} and \emph{any Baire class 2 function can be represented by a double sequence of continuous functions}.  
\end{exa}
\begin{exa}[Uncountable sums]\rm
The concept \emph{unordered sum} is a device for bestowing meaning upon `uncountable sums' $\sum_{x\in I}f(x)$ for any index set $I$ and $f:I\di \R$.  
Whenever $\sum_{x\in I}f(x)$ exists, it must be a `normal' series of the form $\sum_{i\in \N}f(y_{i})$ (see e.g.\ \cite{taomes}*{p.\ xii}); when the antecedent is formulated using the Cauchy criterion of convergence, this fact implies $\NIN$.
This is of historical interest as Kelley notes in \cite{ooskelly}*{p.\ 64} that E.H.\ Moore's study of unordered sums (see \cite{moorelimit2, moorelimit3, moorelimit4}) led to the concept of nets with his student Smith (\cite{moorsmidje}).
Unordered sums can be found in basic or applied textbooks (\cites{hunterapp,sohrab, taomes}) and can be used to develop measure theory (\cite{ooskelly}*{p.\ 79}).  
Tukey develops topology in \cite{tukey1} based on \emph{phalanxes}, a special kind of net with the same structure on the index set as uncountable sums.  
\end{exa}
\begin{exa}[Topology]\rm
The following topological results formulated in third-order arithmetic (see \cite{sahotop}) are connected to $\HBU$ and the Lindel\"of lemma and therefore imply $\NIN$, though we do not have a direct proof of the latter.  
\begin{enumerate}
\item The topological dimension of $[0,1]$ is at most $1$.
\item The Urysohn identity for the dimensions of $[0,1]$.
\item The paracompactness of $[0,1]$ formulated with uncountable coverings. 
\item The existence of partitions of unity for uncountable coverings of $[0,1]$.  
\end{enumerate}
Presumably, many topological notions pertaining to $\R$ depend on its uncountability. 
\end{exa}
\begin{exa}[Separation]\label{XXY}\rm
Separation axioms of the following kind play an important role in RM (see e.g.\ \cite{simpson2}*{I.11.7}):
\[
(\forall n\in \N)(\neg \varphi_{0}(n)\vee \neg\varphi_{1}(n))\di (\exists X\subset \N)(\forall n\in \N)(\varphi_{0}(n)\di n\in X\wedge \varphi_{1}(n)\di n\not\in X).
\]
One readily proves that $\HBU$ is equivalent to this schema for $\varphi_{i}(n)\equiv (\exists f\in 2^{\N})(Y(f, n)=0)$.  Moreover, this schema readily implies $\NIN$ as in the proof of Theorem \ref{exnihilo}.
Indeed, for an injection $Y:[0,1]\di \N$, we cannot have $(\exists x\in A)(Y(x)=n)$ and $(\exists y\in  [0,1]\setminus A)(Y(y)=n)$ at the same time, for any $A\subset [0,1]$.
\end{exa}
\begin{exa}[Connectedness]\rm
A space is \emph{connected} if it is not the sum of two open disjoint sets.  This notion is considered in RM in \cite{simpson2}*{X.1.5} and \cite{brownvi}*{p.\ 193}; the unit interval is mentioned as being connected.  
The connectedness of $[0,1]$ implies $\NIN$ for a general enough notion of open set that includes (i) uncountable unions, and (ii) boolean combinations of uncountable unions that are again open (according to the usual definition).  
\end{exa}
The following two examples pertain to the (fourth order) Lebesgue integral/measure and establish that its very basic properties cannot be proved in $\Z_{2}^{\omega}+\QFAC^{0,1}$.
\begin{exa}[Lebesgue integral]\rm
The Lebesgue integral is well-known and one can derive $\NIN$ from the former's axiomatic formulation as an operator $I:([0,1]\di \R)\di \R$ satisfying the following rather basic properties (assuming $\ACAo$). 
\begin{itemize}
\item For $a, b, c\in \R$, $I(B_{a, b, c})=a\times b$ where $B_{a, b, c}$ is a `box' with height $b$, width $a$, and bottom left corner $(c, 0)$ such that $a+c\leq 1$.  
\item Finite additivity for finite sums of non-overlapping `box' functions. 
\item Dominated convergence theorems for functions as in the previous item. 
\end{itemize}
We may replace the third item by the monotone convergence theorem.  In fact, the dominated and monotone convergence theorems for the Lebesgue integral, as formulated in Bishop's constructive framework (\cite{bish1}*{Ch.\ 6}), also imply $\NIN$.  
\end{exa}
\begin{exa}[Lebesgue measure]\label{frag}\rm
The Lebesgue measure is well-known and one can derive $\NIN$ from the former's axiomatic formulation as an operator $\blambda:([0,1]\di \R)\di \R$ satisfying the following rather basic properties (assuming $\ACAo$). 
Note that we view `subsets of $[0,1]$ as characteristic functions' as in \cite{dagsamVI, kruisje, dagsamVII}.
\begin{itemize}
\item For any $x\in [0,1]$, $\blambda(\emptyset)=\blambda(\{x\})=0$ and $\blambda([0,1])=1$.
\item We have $\blambda(\cup_{n\in \N}E_{n})=0$ if $(\forall n\in \N)(\blambda(E_{n})=0)$.
\end{itemize}
The same result follows if we take the last item together with $\blambda([a,b])=|a-b|$ for $[a, b]\subseteq [0,1]$ and $\blambda(E)\leq \blambda(E\cup [c,d] )$ for $E\cup [c, d]\subseteq [0,1]$.  
We could also replace the last item by \emph{disjoint} countable additivity, a property provable in $\RCA_{0}$ for the second-order approach (\cite{simpson2}*{X.1.6}).
Another suitable property is the `continuity from below' of the Lebesgue measure.  
\end{exa}
\begin{exa}[Lebesgue integral II]\rm
The monotone convergence theorem for the Lebesgue integral is well-known.  The following special case implies $\NIN$ but does not involve the bound from $\Arz$; the conclusion can be stated as in $\WHBU$.
\begin{center}
\emph{For a monotone sequence of Riemann integrable functions $(f_{n})_{n\in \N}$ suppose that $\lim_{n\di \infty}f_{n}(x)=f(x)$ for all $x\in [0,1]$ and $\lim_{n\di \infty} \int_{0}^{1}f_{n}(x)dx$ exists.  Then the Lebesgue integral $\int_{[0,1]}f$ exists.}
\end{center}
In particular, define $g_{n}(x)$ as $1/x$ if $Y(x)=n$ and $x\ne 0$, and $0$ otherwise.  Then $\lim_{n\di \infty} f_{n}(x)=\frac{1}{x}$ for $x\ne 0$ and $f_{n}(x):=\sum_{i=0}^{n}g_{i}(x)$, as for Theorem \ref{arzen}.
\end{exa}
\begin{exa}[Probability theory]\rm
Kolmogorov's three axioms (\cite{kolliek}) of a probability measure $P$ on events $E$ in a sample space $\Omega$ are as follows.
\begin{itemize} 
\item Any event $E$ has a probability in $[0,1]$, i.e.\ $0\leq P(E)\leq 1$.
\item The sample space $\Omega$ satisfies $P(\Omega)=1$. 
\item For $(E_{n})_{n\in \N}$ mutually exclusive events, we have $P(\cup_{n\in \N}E_{n})=\sum_{n=0}^{\infty}P(E_{n})$.
\end{itemize}
In case $\Omega=[0,1]$ and $P(\{x\})=0$ for all $x\in [0,1]$, $\NIN$ follows by noting that for $E_{n}=\{x\in [0,1]:Y(x)=n\}$, we have $1=P([0,1])=P(\cup_{n}E_{n})=\sum_{n}P(E_{n})=0$.
\end{exa}

\begin{exa}[Borel-Cantelli lemma]\rm
The Borel-Cantelli lemma is formulated in e.g.\ \cite{fitzro}*{p.\ 46} as follows:  
\begin{center}
\emph{Let $(E_{k})_{k\in \N}$ be a countable collection of measurable sets with $\sum_{k=0}^{\infty} m( E_{k}) <\infty$. Then almost all $x\in \R$ belong to at most finitely many of the $E_{k}$'s.}
\end{center}
By applying this lemma to $E_{k}=\{x\in [0, 1]: Y(x)=k\}$, we can show that $Y:[0,1]\di \N$ is not an injection. 
\end{exa}
\begin{exa}[Measure and RM-closed sets]\rm
The previous examples pertain to measure theory formulated using higher types, while the following statement is formulated exclusively using `second-order' measure theory. 
In fact, the only higher-order object is the countable collection $A$, as RM-closed sets are represented as sequences of intervals with rational end-points. 
\begin{center}
\emph{For a countable collection $A$ of RM-closed sets in $[0,1]$ with measure zero, $\cup A$ also has measure zero.  }
\end{center}
Note that `$x\in \cup A$' if $x\in E$ for some element $E$ of $A$.
The previous principle readily implies $\NIN$ using the previous arguments.  
\end{exa}

\begin{exa}[Universal theorems]\rm
It is a commonplace that theorems on $\R\di \R$-functions generally only deal with a \emph{sub-class}, e.g.\ all continuous or differentiable functions.   
There are `universal' theorems that apply to \emph{all} $\R\di \R$-functions.
It is easy to show that $\NIN$ follows from \cite{bloemeken}*{Theorem III} as follows.
\begin{center}
\emph{With every function $f( x, y )$ there is associated \(not uniquely, however\) a dense set $D$ of the $XY$ plane such that $f(x, y)$ is continuous, if $(x, y)$ ranges over $D$.}
\end{center}
There are of course more examples of similar, but less basic, theorems.  We believe that E.H.\ Moore's \emph{general analysis} (\cite{moorelimit1, mooreICM}) contains the \emph{first} universal theorems. 
\end{exa}
\begin{exa}[Weak covering lemmas]\rm
There are numerous covering lemmas and related results that imply $\HBU$ or $\LIN(\R)$, as discussed in \cite{dagsamIII}.
The following principle is among the weakest covering lemmas that imply $\NIN$.
\begin{center}
\emph{There are non-identical $a, b\in [0,1]$ such that for any $\Psi:[0,1]\di \R^{+}$ there is a sequence $(x_{n})_{n\in \N}$ such that $[a,b]\subset \cup_{n\in \N}I_{x_{n}}^{\Psi}$.}
\end{center}
One readily derives the latter from the former, which has no first-order strength. 
\end{exa}
\begin{exa}[Weak converence]\rm
Banach's weak convergence theorem from \cite{dies}*{p.\ 405, Theorem 1.2} states the following:
\begin{center}
\emph{Let $(f_{n})_{n\in \N}$ be a uniformly bounded sequence of scalar-valued
functions defined on a set $S$. Then $f_{n}$ converge weakly to zero in the space $B(S)$
of bounded functions on $S$ under the supremum norm \textbf{\textup{iff}} for any sequence $(s_{k})_{k\in \N}$ of points in $S$ we have $\lim_{n\di \infty}\lim_{k\di \infty}f_{n}(s_{k})=0$.}
\end{center}
One derives $\NIN$ from this theorem in the same way as for Theorem \ref{arzen}.
\end{exa}
Finally, a lot can be said about various lemmas due to K\"onig from \cite{koning147}.
\begin{exa}[K\"onig's lemmas]\label{takethat}\rm
As is well-known, $\ACA_{0}$ is equivalent to the statement \emph{every infinite finitely branching tree has a path} (\cite{simpson2}*{III.7.2}).  
We shall refer to the latter as \emph{K\"onig's tree lemma}; Simpson refers to \cite{koning147} as the original source for K\"onig's tree lemma in \cite{simpson2}*{p.\ 125}, but \cite{koning147} does not even mention the word `tree' (i.e.\ the word `Baum' in German).  In fact, the formulation involving trees apparently goes back to Beth around 1955 in \cite{bethweter}, as discussed in detail \cite{wever}.  
K\"onig's original lemmas from \cite{koning147}, formulated there both in the lingo of graph theory and set theory, imply $\NBI$.  
\end{exa}
\begin{exa}[Non-monotone inductive definitions]\rm
The first author has studied the computational properties of the Heine-Borel theorem and the Lindel\"of lemma in relation to \emph{non-monotone inductive definitions} in \cite{dagcie18, dagnmi}. 
The latter notion expresses the iteration of functionals along countable ordinals, which is not easily expressed in weak systems like $\RCAo$.  
The following principle expresses a weak property of non-monotone inductive definitions, namely that there is a fixed point of the operation $\I(F, A):=A\cup F(A) $ that is reached `from below', as follows:
\[\label{braf}
(\forall F:2^{\N}\di 2^{\N})(\exists B\subseteq \N)\big[F(B)\subseteq B\wedge (\forall n\in B)[ n\in B \di (\exists A\subsetneq B)(n\in F(A))\big].
\]
\noindent
It is straightforward to derive $\NIN$ from the previous sentence. 
\end{exa}
For the final example, we need the following rather basic definition.
\begin{defi}[Finite]\label{deadd}\rm
Any $X\subset \R$ is \emph{finite} if there is $N\in \N$ such that for any finite sequence $(x_{0}, \dots, x_{N})$ of distinct reals, there is $i\leq N$ such that $x_{i}\not \in X$.
\edefi
The motivation for this definition of finite set, as opposed to the standard\footnote{In $\ZF$ set theory, a set $A$ is `finite' if there is some bijection to $\{0, 1, \dots, n\}$ for some $n\in \N$; a set $A$ is `Dedekind finite' if any injective mapping from $A$ to $ A$ is also surjective.\label{krukk}} and Dedekind$^{\ref{krukk}}$ definitions, may also be found in the next example.   
\begin{exa}[Weak countability]\label{tarkin}\rm
The following principles readily\footnote{For $A\subset [0,1]$ with $Y:[0,1]\di \R$ injective on $A$, the injection is also a height function.} imply $\NIN$:
\begin{itemize}
\item \emph{for any sequence $(X_{n})_{n\in \N}$ of finite sets, there is $y \in \big([0,1]\setminus \cup_{n\in \N}X_{n}\big)$},
\item \emph{the unit interval is not weakly countable}, 
\end{itemize}
where a set $A\subset \R$ is \emph{weakly countable} if there is a \emph{height} $H:\R\di \N$ for $A$, i.e.\ for all $n\in \N$, $A_{n}:= \{ x\in A: H(x)<n\}$ is finite (Definition \ref{deadd}). 
We note that the notion of `height' is mentioned in e.g.\ \cite{demol}*{p.\ 33} and \cite{vadsiger}.  

\smallskip

As to naturalness, consider the (necessarily countable) set of discontinuities of some function $f:[0,1]\di \R$ of bounded variation (see \cite{dagsamXI}*{\S3.3} for details):
\be\label{bicok}
A:=\{x\in [0,1]:f(x+)\ne f(x-)\},
\ee
where the left and right limits $f(x-)$ and $f(x+)$ have their usual definition. 
The set $A$ as in \eqref{bicok} is readily shown to be weakly countable, say in $\RCAo$; to find an injection from $A$ to $\N$, it seems $\Z_{2}^{\omega}+\QFAC^{0,1}$ does not suffice. 
Similarly, the set 
\be\label{bicok2}\textstyle
A_{k}:=\{x\in [0,1]:|f(x+)- f(x-)|>\frac{1}{2^{k}}\}
\ee
is finite (Definition \ref{deadd}), but we are unable to exhibit even just an injection from $A$ to $\{0, 1, \dots, m\}$ for some $m\in \N$, working in $\Z_{2}^{\omega}+\QFAC^{0,1}$.

\smallskip

In conclusion, \emph{if} one wants to work in a weak logical system, \emph{then} certain finite sets that `appear in the wild', like the set in \eqref{bicok2}, 
are best studied via the notion of finite set as in Definition \ref{deadd}, and not the definition from Footnote~\ref{krukk} involving bijections or injections.
The above notion of weak countability is similarly preferable over the usual definition involving injections or bijections to $\N$. 
\end{exa}

\begin{ack}\rm
We thank Anil Nerode and Pat Muldowney for their helpful suggestions and Jeff Hirst and Carl Mummert for suggesting the principle $\NBI$ to us. 
Our research was supported by the John Templeton Foundation via the grant \emph{a new dawn of intuitionism} with ID 60842 and by the \emph{Deutsche Forschungsgemeinschaft} via the DFG grant SA3418/1-1.
Opinions expressed in this paper do not necessarily reflect those of the John Templeton Foundation.   
\end{ack}

\bye